\documentclass{article}
\usepackage{amsmath,amssymb,amsthm,a4wide}
\usepackage{url}

\usepackage{bbm}
\usepackage{slashed}
\usepackage{enumitem}
\usepackage{mathabx}

\usepackage{comment}
\numberwithin{equation}{section}
\usepackage{endnotes}
\usepackage{graphicx,tikz}
\usepackage[hidelinks]{hyperref}
\usepackage{mathrsfs}
\usetikzlibrary{shapes,arrows}
\usetikzlibrary{calc, positioning, shapes, through, intersections}
\newtheorem{theorem}{Theorem}[section]
\newtheorem*{theorem*}{Theorem}
\newtheorem*{lemma*}{Lemma}
\newtheorem{proposition}[theorem]{Proposition}
\newtheorem{corollary}[theorem]{Corollary}

\newtheorem{lemma}[theorem]{Lemma}

\theoremstyle{definition}
\newtheorem{definition}[theorem]{Definition}

\theoremstyle{remark}
\newtheorem{remark}[theorem]{Remark}

\newcommand{\sumj}{\sum\limits_{j=1,2}}

\newcommand{\cH}{{\mathcal{H}}}

\newcommand{\ord}[1]{_{#1}}

\newcommand{\tper}{\varepsilon^{-1} \Tper}
\newcommand{\ttper}{t_{\rm per}}
\newcommand{\Tper}{T_{\rm per}}
\newcommand{\spec}{{\rm spec}}
\newcommand{\bx}{{\bf x}}
\newcommand{\by}{{\bf y}}
\newcommand{\bxi}{\pmb{\xi}}
\newcommand{\bkappa}{\pmb{\kappa}}
\newcommand{\bv}{{\bf v}}
\newcommand{\bk}{{\bf k}}
\newcommand{\bK}{{\bf K}}
\newcommand{\bm}{{\bf m}}
\newcommand{\bn}{{\bf n}}
\newcommand{\bw}{{\bf w}}

\newcommand{\R}{\mathbb{R}}
\newcommand{\C}{\mathbb{C}}
\newcommand{\RRR}{\mathcal{R}}
\newcommand{\CCC}{\mathcal{C}}
\newcommand{\PPP}{\mathcal{P}}
\newcommand{\tMD}{\widetilde{M}_{\rm D}^\varepsilon}
\newcommand{\usigma}{\underline{\sigma}}
\newcommand{\rev}{}

\newcommand\numberthis{\addtocounter{equation}{1}\tag{\theequation}}

\begin{document}
\title{Effective gaps in continuous Floquet Hamiltonians}
\author{Amir Sagiv$^{a}$ and Michael I. Weinstein$^{b}$ \\
\small $^a$ Department of Applied Physics \& Applied Mathematics, Columbia University,\\ \small New York, NY, USA, \href{mailto:as6011@columbia.edu}{as6011@columbia.edu} ,  \\ \small  $^b$ Department of Applied Physics \& Applied Mathematics, and Department of Mathematics, Columbia University, \\ \small New York, NY, USA, \href{mailto:miw2103@columbia.edu}{miw2103@columbia.edu} ,}


\maketitle
\begin{abstract} 
We consider two-dimensional Schr{\"o}dinger equations with honeycomb potentials and slow time-periodic forcing of the form:
$$i\psi_t (t,\bx) = H^\varepsilon(t)\psi=\left(H^0+2i\varepsilon {\rev   A (\varepsilon t) \cdot \nabla} \right)\psi,\quad H^0=-\Delta +V (\bx)  .$$ 
The unforced Hamiltonian, $H^0$,  is known to generically have Dirac (conical) points in its band spectrum. The evolution under $H^\varepsilon(t)$ of  {\it band limited Dirac wave-packets} (spectrally localized near the Dirac point) is well-approximated on large time scales ($t\lesssim \varepsilon^{-2+}$) by an effective  time-periodic Dirac equation with a gap in its quasi-energy spectrum. This quasi-energy gap is typical of many reduced models of time-periodic (Floquet) materials and plays a role in conclusions drawn about the full system: conduction vs.\ insulation, topological vs.\ non-topological bands.  Much is unknown about nature of the quasi-energy spectrum of the original time-periodic Schr{\"o}dinger equation, and it is believed that no such quasi-energy gap occurs. In this paper, we explain how to transfer quasi-energy gap information about the effective Dirac dynamics to conclusions about the full Schr{\"o}dinger dynamics.
 We introduce the notion of an {\it effective quasi-energy gap}, and establish its existence in the Schr{\"o}dinger model. In the current setting, an effective quasi-energy gap is an interval of quasi-energies  which does not support modes with large spectral projection onto  band-limited Dirac wave-packets.
  The notion of effective quasi-energy gap is a physically relevant relaxation  of the strict notion of quasi-energy spectral gap;
  if a system is tuned to drive or measure at momenta and energies near the Dirac point of $H^0$, then the resulting modes in the effective quasi-energy gap will only be weakly excited and detected.
   \end{abstract}

\section{Introduction}

In this paper we study  time-dependent Schr{\"o}dinger equations of the form:
\begin{equation}
 i\partial_t\psi =  \left( H^0+ W(t,\bx,-i\nabla) \right)\psi,
 \quad H^0=-\Delta+V(\bx)\ .
 \label{gen-model}\end{equation} 
Here,  the potential, $V(\bx)$,  models a two-dimensional medium with the symmetries of a honeycomb tiling of the plane. Such {\it honeycomb lattice potentials},  $V$, are real-valued, so that $H^0$  is self-adjoint and periodic with respect to the equilateral triangular lattice in $\mathbb{R}^2$; see Section \ref{honey} for a detailed discussion. We assume that the operator $W(T,\bx,-i\nabla)$ is $\Tper$-periodic with respect to $T$, $\Lambda$-periodic with respect to $\bx$, and self-adjoint with domain which is a subset of the domain of~$H^0$ for each $T$. Naturally occurring and engineered material systems governed 
  by models such as  \eqref{gen-model} are referred to as  {\it Floquet materials} \cite{hone1997time, nathan2015topo, ozawa2019topological}. See Section \ref{2D-mat}, where we discuss two physical settings, in condensed matter physics and in photonics, where the class of models \eqref{gen-model} arises.

Many of the important properties of graphene and related 2D materials are intimately related to {\it Dirac points} (conical band degeneracies)
 in the band structure of $H^0$. 
{\it The goal of this paper is to explore, in the context of \eqref{gen-model}, the effects of time-periodic driving on the dynamics of wave packets which are spectrally concentrated near Dirac points.}

\subsection{2D materials, honeycomb structures and Dirac points}\label{2D-mat}

 Two-dimensional materials
are  of great current interest in fundamental and applied science. The paradigm is graphene, a macroscopic single atomic layer of carbon atoms, centered on the vertices of a honeycomb lattice \cite{RMP-Graphene:09}. Single electron models of graphene, both the tight-binding (discrete)  \cite{Wallace:47} and the continuum Schr{\"o}dinger operator with a continuum honeycomb lattice potential $H^0=-\Delta+V$ \cite{FW:12}, have conical degeneracies at distinguished quasi-momentum
in the band structure. These {\it Dirac points} arise due to the symmetries of the honeycomb Schr{\"o}dinger operator. The envelope of a wave-packet, which is spectrally concentrated near a Dirac point, evolves  according to a two-dimensional Dirac equation on large time scales \cite{FW:14}. Due to the zero density of states at the Dirac energy, graphene is referred to as a semi-metal.

For time-independent Hamiltonians, opening a gap in the spectrum by breaking spatial symmetries can be leveraged to induce states with desired localization properties and energies in the gap. Spatially localized defect perturbations of such ``gapped Hamiltonians'' give rise to defect modes which are pinned to the location of the defect \cite{Figotin-Klein:97, Figotin-Kuchment:96b, hoefer2011defect, Joannopoulos}, while line-defects in the direction of periodic lattice vector of the bulk 
 vector give rise to edge states, which are plane-wave like in the direction of the line-defect and which decay transverse to it;
see e.g., \cite{delplace2011zak, FLW-2d_edge:16, Graf-Porta:13, mong2011edge,Karpeshina:97}. For line defects in systems which break $\mathcal{C}$ (Hamiltonians that do not commute with complex conjugation) the induced bulk energy gap is filled with energy spectrum; see e.g., \cite{Drouot:19,Drouot:21, drouot2020edge, LWZ:18}, a phenomena which can be explained by non-trivial topological indices; see, for example, 
  \cite{kellendonk2002edge, thiang2016k}.\\
  Time-periodic driving is another mechanism for breaking symmetries of the bulk and opening spectral gaps, and the corresponding  topological indices are typically defined in systems with gapped quasi-energy spectrum \cite{asboth2014chiral, nathan2015topo, rudner2013anomalous, rudner2020band}.
  \footnote{ An exception is the case of {\it mobility gaps} in strongly disordered discrete systems \cite{shapiro2019strongly}.}

 

 The class of time-periodically driven PDEs \eqref{gen-model} arises in physical settings, such as:
\begin{enumerate}[label=(\alph*)]
\item the modeling of a  graphene sheet, excited by a time-varying electric field  \cite{perez2014floquet, wang2013observation}. Here, $H^0=-\Delta+V$ is a single-electron Hamiltonian for graphene and time-dependence in $H^\varepsilon(t)$ models the excitation of the graphene sheet by an external electro-magnetic field. {\rev In this work, we consider a vector potential which is space-independent, and which by Maxwell's equations, induces a homogeneous time-periodic electric field; see \cite{krieger1986time}.}

\item 
the propagation of light in a hexagonal array of  helically coiled optical fiber waveguides \cite{ozawa2019topological, Rechtsman-etal:13}. Here, the Schr{\"o}dinger equation describes the propagation in the time-like longitudinal direction of a continuous-wave (CW) laser beam propagating through a hexagonal or triangular transverse array of optical fiber waveguides. Beginning with Maxwell's equations, under the nearly monochromatic and paraxial approximations, one obtains \eqref{gen-model} for the longitudinal evolution of the slowly varying envelope of the classical electric field. Suppose the fibers are longitudinally coiled. Then, in a rotating coordinate frame, we obtain \eqref{gen-model}, where $V$ models the uncoiled fiber-array, and the time-periodic perturbation, $W$, captures effect of periodic coiling.

\end{enumerate}

Other fields in which time-periodic modulations are applied to spatially periodic materials include acoustics, plasmonics, and mechanical metamaterials; see, e.g., \cite{fleury2016floquet, nash2015topological, wilson2019temporally} and references therein. In both settings (a) and (b) the operator $W(T,\bx,-i\nabla)$, while self-adjoint, does not commute with $\mathcal{C}$ and hence  \eqref{gen-model} does not have time-reversal symmetry. As in the case of time-independent  (autonomous) Hamiltonians, this is a source
  of topological phenomena.
  %
 %
Such physical systems are therefore called {\it Floquet topological insulators}.  Time-periodic Hamiltonians modeling  Floquet materials have many interesting phenomena, even more varied than their time-invariant analogues.

\subsection{Hamiltonians for Floquet materials and the monodromy operator}\label{flo}

Consider a general non-autonomous Hamiltonian system $i\partial_t Z=H(t)Z$, where for each $t$: 
\[\textrm{$H(t+\ttper)=H(t)$ and
 $H(t)$ is a self-adjoint operator acting in the Hilbert space, $\cH$.}\]
   Denote by $U(t)$ the unitary evolution operator which maps the data $Z(t=0)=Z_0\in\cH$ to the solution $Z(t)=U(t)Z_0\in \cH$ for some $t\geq 0$.
  The dynamics are characterized by  operator, $M=U(\ttper)$, since $Z(n\ttper)=M^nZ_0$,\ for all $n\ge0$.  $M$ is called the monodromy operator,  and since $U(t)$ is unitary, its spectrum is constrained to the unit circle in $\mathbb{C}$.   For $\lambda\in \spec(M)$ we write $\lambda=e^{i\mu}$ and call $i\mu$ (modulo $2\pi i$) the associated Floquet exponent.   
  If $MZ_\mu = e^{i\mu} Z_\mu$, then $p_\mu(t)\equiv e^{-i\frac{\mu}{\ttper}t} U(t)Z_\mu $ is $\ttper$-periodic and satisfies
 the eigenvalue problem for the {\it Floquet Hamiltonian} \begin{equation}\label{eq:Floq_hamil}
    \mathcal{K}(t)p(t)\ \equiv\ \left(\ i\partial_t-H(t)\ \right) p(t) = \frac{\mu}{\ttper} p(t) .
   \end{equation}
   Thus, $e^{i\mu}\in\spec(M)$ if and only if $\frac{\mu}{\ttper}$ is in the spectrum of $\mathcal{K}$ acting in $L^2(S^1;\mathcal{H})$. The  quantity $\frac{\mu}{\ttper}$ is called a {\it quasi-energy}. For autonomous Schr{\"o}dinger operators ($H(t)$ independent of $t$) the quasi-energies coincide with the eigenvalues of $H_0$,  modulo $2\pi \mathbb{Z}/ t_{\rm per}$.

 \bigskip

  \subsection{What this article is about}\label{what}
  
    We focus on the evolution of wave-packets of the special case of \eqref{gen-model} with 
  \begin{equation}
  W(t,\bx,-i\nabla)=2i\varepsilon A(\varepsilon t)\cdot \nabla,\quad 0<\varepsilon\ll1
  \label{Wdef}
 \end{equation}
{\rev  where $A(T):[0,\Tper)\to \R^2$ is continuous, $\bx$-independent, and $\Tper-$ periodic.} Hence, \eqref{gen-model} becomes
  \begin{align}\label{eq:lsA}
i\partial_t\psi (t,\bx) &= H^\varepsilon(t) \psi 
,\qquad H^\varepsilon(t)\equiv -\Delta + V(\bx) + 2i\varepsilon A(\varepsilon t)\cdot \nabla \ .
\end{align}
where $V$ is taken to be a honeycomb lattice potential; see Section \ref{honey}.  For $\varepsilon\ne0$,  the operator  $H^\varepsilon(t)$ is self-adjoint but does not commute with  $\CCC$. Since $0<\varepsilon\ll1$,  we are focusing on the regime of a {\em slowly varying} time-periodic perturbation.
{\rev To the best of our knowledge, the only previous analytical study of the slow time-periodic regime for \eqref{eq:lsA} is in \cite{ablowitz2015adiabatic}. This study focuses on the tight-binding regime for the bulk potential, and provides asymptotic analyses of linear and nonlinear edge modes in the regime where interactions with ``higher bands'' can be neglected. Our analysis makes no restrictive assumptions on the asymptotic regime of the bulk Hamiltonian
 and our study focuses on consequences of time-forcing induced interactions of the ``Dirac bands'' with bands which are distant from the Dirac point.} The case of rapidly varying time-forcing is studied, for example, in \cite{bal2021multiscale, guglielmon2018photonic, perez2015hierarchy}. 


Denote the solution of the initial value problem for \eqref{eq:lsA} with data
 $\psi(0,\cdot)=\psi_0\in H^s(\mathbb{R}^2)$, $s\ge0$, by 
 \[ \psi^\varepsilon(t)=U^\varepsilon(t)\psi_0\quad {\rm or}\quad \psi^\varepsilon(x,t)= U^\varepsilon[\psi_0](x,t).\]
 For generic honeycomb lattice potentials, $V$, there exist bands (of the unperturbed ($\varepsilon=0$) Hamiltonian $H_0$) which touch at 
{\it Dirac points}: quasi-momentum / energy pairs $(\bk_{\rm D},E_D)$ at which exactly two dispersion surfaces touch conically; see the discussion in Section~\ref{honey}. 

By exploiting the multi-scale character of \eqref{eq:lsA}, we derive  a homogenized, time-periodically forced effective Dirac Hamiltonian, $\slashed{D}(T)$, which governs the evolution of {\it Dirac wave-packets} envelopes. On a closed and invariant subspace
 of the Hilbert space, the monodromy (unitary) operator $M_{\rm D,d_0}$ associated with $\slashed{D}(T)$
 has a gap, i.e., an arc on $S^1$ with no spectrum. This closed and invariant subspace corresponds,
  in the $H^\varepsilon(t)$ (un-approximated) dynamics, to the physically
 interesting situation of {\it band-limited Dirac wave-packets}, {\it i.e.} those
  built from Floquet-Bloch modes of $H^0$, whose energies and quasi-momentum components are near a Dirac point. 
  
{\it We prove (Theorem \ref{thm:quasigap_general}) that any $L^2(\mathbb{R}^2)$ wave-packet comprised of modes \underline{within} the Floquet multiplier (quasi-energy) gap of the effective Hamiltonian, $\slashed{D}(T)$,
 is dominated by spectral components with corresponding energies and quasi-momenta bounded \underline{away} from the Dirac energy, $E_D$.}
Such wave-packets therefore  spatially oscillate on different length-scales from the Dirac modes. 
We call the arc in $S^1$ of such Floquet exponents, and the corresponding quasi-energy interval, an \underline{\it effective spectral gap}. This is a relaxation of the usual notion of spectral gap. 
 The  nature of the spectrum in the effective gap, and of $M^{\varepsilon}$ in general, is a difficult open problem; 
see the discussion in Section \ref{previous}. We conjecture that the spectrum of the monodromy operator of $H^\varepsilon(t)$ covers the entire unit circle in $\mathbb{C}$. Without relying on detailed information of the spectral measure, 
Theorem~\ref{thm:quasigap_general} provides information on the modes associated with the effective gap.

From a physical perspective, there are two main implications of Theorem \ref{thm:quasigap_general}, which justify the term ``effective gap''. We discuss this in terms of the static band structure ($H^0 = -\Delta +V$), assumed to have well-separated bands. One may think in terms of the system's  inputs
and outputs at some finite time. Consider a measuring device with sensitivity tuned to a neighborhood of the Dirac energy, $E_D$. If the parametrically forced system generates a mode with {\em quasi-energy} inside the effective gap, since it consists mainly of energies (with respect to $H^0$) far from the Dirac energy (Theorem \ref{thm:quasigap_general}), this mode will only be very weakly detected.
On the other hand, in terms of input excitations to the system, suppose the initial excitation is a Dirac wave-packet, whose energy-spectrum is localized
 near a Dirac point (Proposition \ref{prop:wp_dk}). Because of the approximation of the Schr{\"o}dinger  evolution  by the effective Dirac equation (Theorem \ref{thm:valid}), on the time-scale of the forcing period, such a mode will remain energetically localized near a Dirac point. Therefore, modes with {\em quasi-energy} in the effective gap, which are dominated by components from distant  energies (Theorem \ref{thm:quasigap_general}), can only be very weakly excited. Summarizing: if the system is energetically tuned to the Dirac point, either on the input or output side, then it will effectively behave as an insulator, with only weak excitations inside the effective gap. Theorem \ref{thm:quasigap_general} goes further than that; it provides {\em quantitative} information on the size of these excitations.

There are several steps along the way to proving Theorem \ref{thm:quasigap_general}. 
\begin{enumerate}
\item Theorem \ref{thm:valid}:\ The time-evolution \eqref{eq:lsA} for wave packet initial data, which are spectrally concentrated about a Dirac point, so-called {\it Dirac wave packets} (see \eqref{eq:slow_wp} and Section \ref{sec:bl}), is  governed by 
a time-dependent effective Dirac equation, for time scales of order $\mathcal{O}(\varepsilon^{-2+})$; Proposition \ref{prop:wp_dk}.
\item Corollary \ref{cor:mono_approx}: The time scale of validity of the effective Dirac equation is much larger than the period of temporal forcing, $\Tper\varepsilon^{-1}$. Hence, we can approximate the monodromy operator for \eqref{eq:lsA}, $M^\varepsilon$, which acts in $L^2(\mathbb{R}^2)$ in terms of an effective monodromy operator $M_{\rm Dir}$, acting on $L^2(\mathbb{R}^2;\mathbb{C}^2)$.
\item Proposition \ref{prop:Dirac_gap}: We prove that  when acting on an invariant subspace of band-limited initial data, $M_{\rm Dir}$ has a quasi-energy gap.  See Figure \ref{fig:threeSpectra}.

\item The separation of scales in the dynamical system \eqref{eq:lsA} allows application of a homogenization / averaging lemma (Lemma~\ref{lem:homog}) to 
``carry back''  the spectral gap of the effective Dirac monodromy operator, $M_{\rm Dir}$, to an effective gap of  $M^{\varepsilon}$. 
\end{enumerate}

\subsection{Previous works and remarks}\label{previous}

\paragraph{Spectral theory of parametrically forced Hamiltonians. }Since $H^\varepsilon(t)$ is invariant under translations in $\Lambda$, Floquet-Bloch theory \cite{Kuchment:12} reduces the  spectral properties of the Floquet Hamiltonian $\mathcal{K}$ (see \eqref{eq:Floq_hamil}) in the space $L^2(S^1;L^2(\mathbb{R}^2))$ 
to its action on the family of  subspaces $L^2(S^1;L^2_{\bk})$,  where $L^2_\bk$ denotes the space of $\bk-$ pseudo-periodic functions on $\R^2$ and $\bk\in\mathcal{B}$, the Brillouin zone associated with  $\Lambda$:
\[ L^2(S^1;L^2(\R^2))\ \textrm{spectrum of}\ \mathcal{K}\ =\  \textrm{Union over $\bk\in\mathcal{B}$ of the }\ \ L^2(S^1;L^2_\bk)\ \textrm{spectra of}\ \mathcal{K}
  \]
Spectral problems of this latter type, which correspond to time-periodically forced wave equations on the spatial torus, have been explored in the deep technical works \cite{bambusi2001time, bambusi2017reduce, eliasson2008reducibility,  feola2020reducibility, howland1989floquetII, montalto2021linear}.  These results focus on establishing the existence of point spectra of $\mathcal{K}(t)$ under strong growth assumptions on the eigenvalue spectrum of the unforced wave equation. The nature of the $L^2(S^1;L^2_\bk)$ spectrum of $\mathcal{K}$ when these growth conditions are violated is an open problem~\cite{montalto2021linear}.

The nature of the $L^2(S^1;L^2(\mathbb{R}^2))$ spectrum of 
$\mathcal{K}$ and the corresponding $L^2(S^1;L^2(\mathbb{R}^2))$ spectrum (a subset of  $S^1$) of the monodromy operator are also open problems. That we can expect the latter to cover the unit circle can be understood, heuristically, via the mechanism physicists refer to as {\it band folding}. This intuition is based on high-energy asymptotics of the unforced problem; for any fixed $\bk\in\mathcal{B}$,  the high quasi-energy bands (via Weyl asymptotics) are approximated by the eigenvalues of $-\Delta$ on $L^2(\R^2/\Lambda)$, modulo $2\pi\mathbb{Z} / \Tper$, which for typical $\Tper$ are dense in $S^1$; see Figure \ref{fig:fold}. Certainly 
 the union over all $\bk \in \mathcal{B}$ would be expected to be dense as well. 
  We justify this argument in the context of our parametrically forced effective Dirac operator, $\hat{\slashed{D}}(T)$, in Remark \ref{rem:BL} and Section \ref{sec:wkb}. 
  
 \paragraph{Spectral gaps in other reduced models.} Effective (approximate) models are often used to provide analytically or computationally tractible precise descriptions in specified asymptotic regimes. One class consists of homogenized operators, such as our effective Dirac operator, is at the center of this work; see Section \ref{sec:Dirac} and \cite{FW:14}.
Another class consists of spatially discrete and periodic in time tight-binding effective models of crystalline Floquet materials, 
 derived for example, via Magnus (high-frequency) expansion \cite{blanes2009magnus, bukov2015high}, governing low-lying mode-amplitudes
 of the unforced problem. Here,  $\mathcal{K}$ acts in $l^2(S^1_{\Tper};\mathcal{H})$, where $\mathcal{H}=l^2(\mathbb{G};\mathbb{C}^N)$ with $\mathbb{G}$
  being a translation invariant spatial lattice and $N$ the number of degrees of freedom per unit cell. The system has $N$ Floquet-exponents bands defined over the Brillouin zone (graphs of the map from $\bk$ to the eigenvalues of the monodromy over $L^2_{\bk}$).	
  For the case of graphene $\cH=l^2(\mathbb{H};\mathbb{C}^2)$, where $\mathbb{H}$ denotes the set of honeycomb vertices in $\mathbb{R}^2$ and $N=2$  is the number of carbon atoms per unit cell.  Such models can exhibit a spectral gap on the unit circle in the spectrum
   of their monodromy operator
   \cite{ablowitz2017tight, RMP-Graphene:09, Rechtsman-etal:13}.
   
   \paragraph{The significance of gaps.} As in the case of time-independent Hamiltonians, continuous spectra of quasi-energies are associated with energy propagation (``conduction'') and spectral gaps with non-propagation (``insulation''). In the Floquet systems which arise in condensed-matter physics, Kubo-type formulae, analogous to the autonomous case, are used to quantify conductance \cite{dehghani2015out, oka2009photovoltaic, wackerl2020floquet,bal2021multiscale}. For Floquet systems, topological indices, closely related to the Kubo formula, have been rigorously defined if there is a spectral gap; see, for example, \cite{graf2018bulk} and \cite{rudner2013anomalous}. 

\begin{figure}[h]
\centering
\includegraphics[scale=.6]{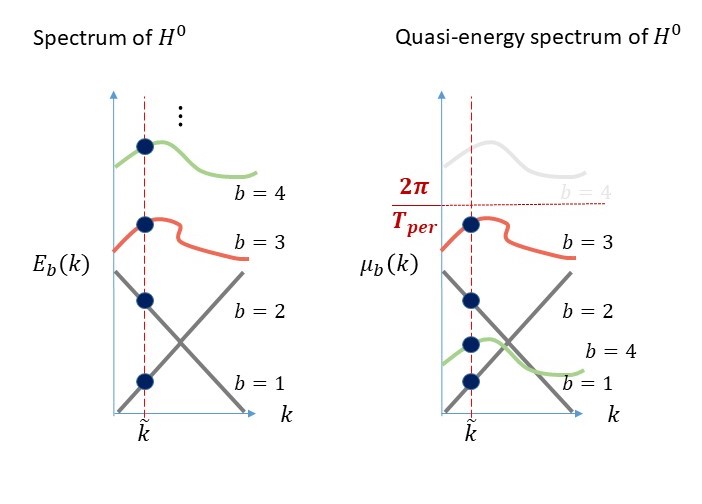}
\includegraphics[scale=.6]{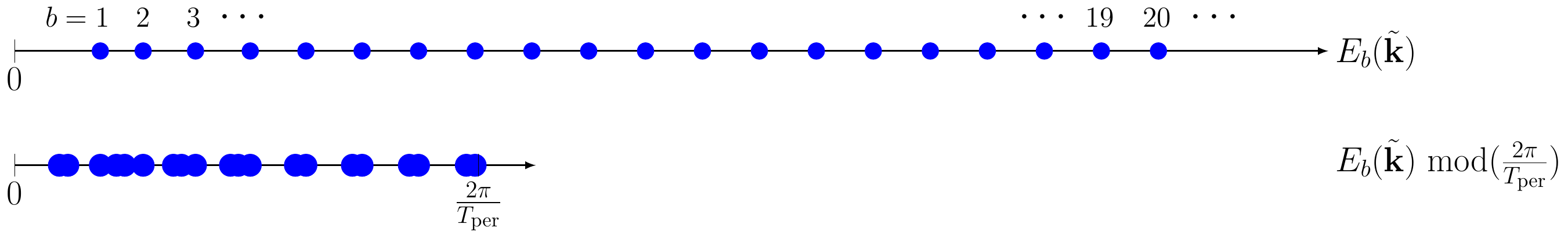}
\caption{ Band folding heuristics.  Top left: Schematic of the band spectrum of $H^0$. Lowest two bands touch conically at a Dirac point. Top right: Quasi-energy spectrum of $H^0$, viewed as  $\Tper$ time-periodic. The dispersion curve $E_4(\bk) > 2\pi /\Tper$. Modulo $ (2\pi/\Tper) $ the $\bk\mapsto E_4(\bk)$ is translated to the quasi-energy range $[0, 2\pi /\Tper]$, and the spectrum fills in. 
Middle: For a fixed $\tilde{\bk} \in \mathcal{B}$, we mark $E_1(\tilde{\bk}),\ldots, E_{20}(\tilde{\bk})$. Bottom:  The points $\{E_b(\tilde{\bk}) ~ {\rm mod}(2\pi/\Tper)\}_{b\geq 1}$ are expected to be dense in the quasi-energy range $[0, 2\pi /\Tper]$. 
}
\label{fig:fold}
\end{figure}



\subsection{Structure of the paper}
The remainder of the paper is organized as follows: In Section \ref{honey} we review relevant preliminaries in Floquet-Bloch and the spectral theory of honeycomb potentials. In Section \ref{eff-Dirac} we present the results on the approximation of the  dynamics governed by the Schr{\"o}dinger Hamiltonian, $H^\varepsilon(t)$, in terms of the effective Dirac Hamiltonian, $\slashed{D}(T)$. We also discuss the quasi-energy  gap property of the Dirac dynamics on the space of  band-limited functions. In Section~\ref{sec:bl} we discuss the properties of Dirac wavepackets. In Section \ref{sec:mainres} we present the effective gap result (Theorem \ref{thm:quasigap_general}), the main result of this paper. We present the proofs of the main results in the subsequent sections; the proof of the main result in Section \ref{sec:mainpf}, of the homogenization / averaging lemma (Lemma \ref{lem:homog}) in Section \ref{sec:pf_homog}, the derivation and proof of validity of the effective Dirac dynamics (Theorem \ref{thm:valid}) in Section \ref{sec:Dirac}, and of the spectral characterization of Dirac wave-packets (Proposition \ref{prop:wp_dk}) in Section~\ref{sec:proj}. Finally, in Section \ref{sec:wkb} we use a WKB expansion to  justify Remark \ref{rem:BL}, which states that the effective Dirac Hamiltonian, $\slashed{D}(T)$, has no quasi-energy gap .

\subsection{Notation and conventions}\label{notation}

\begin{itemize}
\item Triangular lattice: $\Lambda = \mathbb{Z}\bv_1\oplus \mathbb{Z}\bv_2 $, where 
\begin{equation}
\bv_{1} = \begin{pmatrix}\sqrt{3}/2 \\ 1/2\end{pmatrix},\quad \bv_{2} = \begin{pmatrix}\sqrt{3}/2\\ - 1/2\end{pmatrix}.\label{bv-def}
\end{equation}
Here, $\Omega \subset \mathbb{R}_\bx^2$ denotes the fundamental cell; see left panel of Fig.\ \ref{fig:hc}.
\item $\Lambda^* = \mathbb{Z}\bk_1\oplus \mathbb{Z}\bk_2$ is the dual lattice, where 
\begin{equation}  \bk_{1} =  \frac{4\pi}{\sqrt3}\begin{pmatrix} 1/2 \\ \sqrt3/2\end{pmatrix},\quad \bk_{2} =  \frac{4\pi}{\sqrt3}\begin{pmatrix} 1/2 \\ -\sqrt3/2\end{pmatrix} .\label{bk-def}\end{equation}
 The Brillouin zone, $\mathcal{B}\subset \left(\mathbb{R}_\bx\right)^*=\mathbb{R}_\bk^2$, is the dual fundamental cell.
\item $\bk_{\rm D} = \bK,\bK^\prime$ are the vertices of the Brillouin zone $\bK=(1/3)(\bk_1-\bk_2)$ and $\bK^\prime = -\bK$; see right panel of Figure \ref{fig:hc}.
\item The Pauli matrices are 
\begin{equation}
\sigma_1 = \left(\begin{array}{ll}
 0 &1 \\1 &0
\end{array} \right) \, , \quad  \sigma_2 = \left(\begin{array}{ll}
 0 &-i \\i &0
\end{array} \right) \, , \quad \sigma_3 = \left(\begin{array}{ll}
 1&0 \\ 0 &-1
\end{array} \right) \, \label{pauli}\end{equation}
\item $V(\bx)$ is a honeycomb lattice potential; see Section \ref{hlp-dp}.
\item $C(S^1_{\Tper};\mathbb{R}^2)$, with $S^1_{\Tper} = \mathbb{R}/(\Tper\mathbb{Z})$ is the space of $\mathbb{R}^2-$
vector-valued continuous functions, which are $\Tper$ periodic. 
\item $\chi(|s|< a)=$ the indicator function of the set $\{s:|s|<a\}$.
\item  For $f\in L^2(\mathbb{R}^2)$, the Fourier transform is denoted by
\begin{align*}
&\widehat{f}(\xi) = \mathcal{F}[f](\xi)= \int_{\mathbb{R}^2} e^{-i\xi\cdot X} {f}(X)dX \, ,\\
&\chi(|\nabla|<a)f =\frac{1}{(2\pi)^2}\int_{\mathbb{R}^2} e^{i\xi\cdot X}  \chi(|\xi|<a)\widehat{f}(\xi)d\xi \, ,\\
&\chi(|\nabla|<a) L^2(\mathbb{R}^2) = \{ \chi(|\nabla|<a)f\ :\ f\in L^2(\mathbb{R}^2) \} = 
\{f\in L^2(\mathbb{R}^2) : {\rm supp}(\hat{f})\subset B_a(0)\} \, .
\end{align*}
\item The spectrum of an operator, $L$, is denoted $\spec(L)$.
\item Let $H(t)$ be a self-adjoint time-dependent Hamiltonian. For the Schr{\"o}dinger equation $i\psi_t = H(t)\psi$, $\psi(0)=\psi_0$, we denote the unitary flow map by $U(t)$, {\it i.e.} $\psi(t)=U(t)\psi_0$.  
If $t\mapsto H(t)$ is $\Tper$ periodic, then we denote the monodromy operator  by  $M=U(\Tper)$. For the right hand side of \eqref{eq:lsA}, we denote the flow by $U^{\varepsilon}$ and the monodromy by $M^{\varepsilon}$.

\item For a unitary operator $U$ on a Hilbert space $\mathcal{H}$, the spectral projection-valued measure $\Pi$ on the Borel $\sigma$-algebra of the unit circle $S^1$ satisfies $$U\ =\ \int\limits_{S^1} z \, d\Pi(z)  .$$

\item {\rev The Sobolev norm $\|\cdot \|_{H^s(\R ^2)}$ for $s\in \mathbb{N}$ is defined as 
$\|f\|_{H^s(\R ^2)}^2 \equiv  \sum\limits_{|\alpha| \leq s} \| \partial^{\alpha} f\|_{L^2(\R ^2)}^2  \, .$}

\end{itemize} 

\subsection{Acknowledgments}  The authors thank M.\ Rechtsman, J.\ Guglielmon. S.\ Tsesses, and J.\ Shapiro for stimulating discussions.
M.I.W. was supported in part by US National Science Foundation grants DMS-1620418 and DMS-1908657, and Simons Foundation Math + X Investigator Award \#376319.

\section{Honeycomb potentials and Dirac points}\label{honey}

We give a brief review of spectral theory of periodic elliptic operators, and  Dirac points for honeycomb Schr{\"o}dinger operators \cite{Eastham:74,Kuchment:16,RS4,FW:12}.

\subsection{ Review of Floquet-Bloch theory}
Consistent with the notation of \eqref{eq:lsA}, we set:
$$ H^0 = -\Delta +V(\bx) \, , \qquad {\rm acting}~{\rm on}~L^2(\mathbb{R}^2) \, ,$$
where $V$ is real-valued and periodic with respect a lattice 
$\Lambda = \mathbb{Z} \bv_1 \oplus \mathbb{Z} \bv_2 \subset\R^2$. The associated dual lattice is $\Lambda^* = 
\mathbb{Z} \bk_1\oplus \mathbb{Z}\bk_2 $ and $\bk_m \cdot \bv_n  = 2\pi\delta_{nm}$; see \eqref{bv-def}, \eqref{bk-def}.
  Let $\Omega\subset \mathbb{R}_\bx^2$ denote a fundamental cell for  $\mathbb{R}_\bx^2/\Lambda$ 
   and  $\mathcal{B}$, the {\it Brillouin zone},  denote the fundamental cell for $ \mathbb{R}_\bk^2/\Lambda$. 
     For each quasi-momentum (crystal momentum) $k\in \mathcal{B}$, denote by $L^2_\bk$ the space of $\bk$-pseudoperiodic functions 
$$L^2_\bk \equiv \{ u \in L^2_{\rm loc}(\mathbb{R}^2) ~~ {\rm such}~{\rm that}~~ u(x+\bv) = e^{i\bk\cdot \bv}u(x) \, , \quad \bv\in \Lambda   \} \, .$$
The space $L^2(\mathbb{R}^2)$ admits the fiber decomposition $L^2(\mathbb{R}^2) = \int^\oplus_{\mathcal{B}} L^2_\bk \, d\bk$. Since $H^0$ is translation invariant with respect to $\Lambda$,  it has the fiber decomposition: $H^0 = \int^\oplus_{\mathcal{B}} H_\bk^0 \, d\bk$, where $H_\bk^0=H^0\large|_{L^2_\bk}$, the self-adjoint operator $H^0$ acting in $L^2_\bk$, has compact resolvent and therefore has an infinite sequence of finite multiplicity real eigenvalues, tending to infinity,  
\[ E_1(\bk)\le E_2(\bk)\le \dots E_b(\bk)\le \dots ,\]
listed with multiplicity, with corresponding eigenmodes $\Phi_b(x;\bk)\in L^2_\bk$, known as Bloch modes, which satisfy
\[ H^0\Phi_b (x;\bk) = E_b (\bk) \Phi_b (x;\bk),\quad \Phi_b(\cdot;\bk)\in L^2_\bk\ .\]
The maps $\bk\in\mathcal{B}\mapsto E_b(\bk)$ are Lipschitz continuous. 
 The two-dimensional surfaces $E_b(\bk)$ are called the {\it dispersion surfaces} $H^0$. We refer to the collection of 
  all pairs $(E_b(\bk),\Phi_b(x,\bk))$, where $b\ge1$ and $k\in\mathcal{B}$, as the {\it band structure} of $H^0$. 
  
\subsection{Honeycomb lattice potentials and Dirac points}\label{hlp-dp}

Here and henceforth, $\Lambda$ denotes the equilateral triangular lattice in $\R^2$; see \eqref{bv-def}. A sufficiently regular function $V$ is a {\it honeycomb  potential} if  $V$ is
 real-valued,  $\Lambda$-periodic, even and $2\pi/3$--rotationally invariant; see \cite[Definition 2.1]{FW:12}: 
\begin{subequations}\label{symm-def}
\begin{equation}
 [\CCC,V(\bx)]=0 \, ,\quad [\PPP ,V(\bx)]=0 \, ,\quad  
  [\RRR,V(\bx)]=0 \, , \ \ \ \  \text{ where} \\
  \end{equation}
  \begin{equation}
\CCC[f](\bx) \equiv \overline{f(\bx)}\, ,\quad \PPP[f](\bx) \equiv f(-\bx) \, ,\quad \RRR[f](\bx) \equiv f(R^*\bx)\ \, .
\end{equation}
\end{subequations}
An example of a honeycomb potential is a two dimensional infinite array of ``atomic potential wells'' centered on the vertices of a triangular or honeycomb lattice; see \cite[Section 2.3]{FW:12}. The honeycomb case corresponds to the single electron model of graphene; see Figure \ref{fig:hc}.

A {\it Dirac point} of $H^0$ is a quasi-momentum / energy pair, $(\bk_{\rm D},E_{\rm D})$, where two consecutive dispersion surfaces touch 
in a right circular cone; there exists $v_{\rm D}>0$ such that,  as $|\bk-\bk_{\rm D}|\to0$,
 \begin{equation}
 E_{\pm}(\bk) = E_{\rm D} \pm v_{\rm D} |\bk-\bk_{\rm D}|\cdot \left(1+o|\bk-\bk_{\rm D}|\right)\ .
 \label{D-cone}
 \end{equation} 
 Associated with a Dirac point, $(\bk_{\rm D},E_{\rm D})$, is an eigenvalue of multiplicity two of $H^0$ acting in  $L^2_{\bk_{\rm D}}$.

In \cite{FW:12} (see also \cite{FLW-MAMS:17,FLW-CPAM:17,LWZ:18}) it is proved that for generic  honeycomb potentials, the band structure of $H^0$ contains Dirac points, $(\bk_{\rm D},E_D)$,
 where $\bk_{\rm D}$ varies over the six vertices   of the Brillouin zone $\mathcal{B}$, the so-called {\it high symmetry quasi-momenta}. There are exactly two independent high symmetry quasi-momenta, designated $\bK$ and $\bK^\prime$; a choice with $\bK^\prime=-\bK$ is shown in Figure~\ref{fig:hc}.

\begin{figure}[h]
\centering
\includegraphics[scale=1]{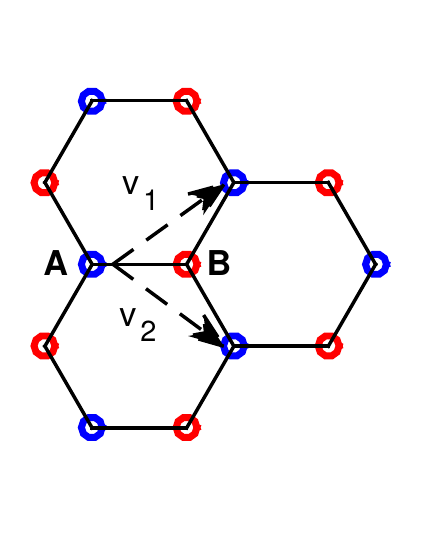}\hspace{2cm}
\includegraphics[scale=1]{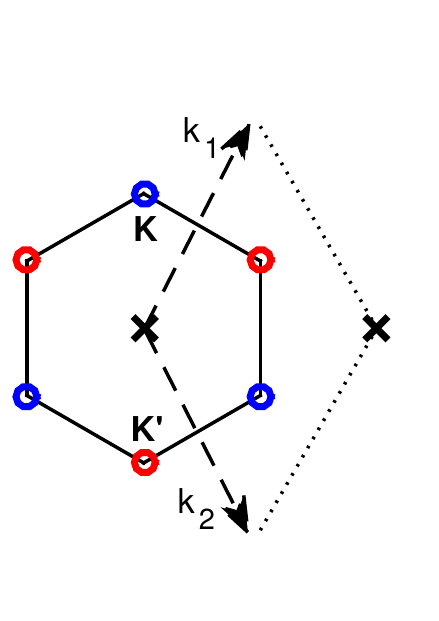}
\caption{Left: Part of the honeycomb lattice in $\R^2$. Indicated are basis vectors $\bv_1$ and $\bv_2$  of the equilateral triangular lattice, $\Lambda$. An example of a honeycomb potential (Section \ref{hlp-dp}) is one  consisting of radially symmetric wells centered at vertices of the honeycomb \cite{FW:12}. Right: Brillouin zone $\mathcal{B}$, high-symmetry points $\bK$ and $\bK'$ and basis vectors, $\bk _1$ and $\bk _2$, of dual lattice $\Lambda^*$.
}
\label{fig:hc}
\end{figure}

  \noindent{\bf N.B.} {\it Throughout this paper we will focus on a Dirac point at $(\bk_{\rm D},E_D)=(\bK,E_D)$. The results for the Dirac point $(\bK^\prime,E_D)$  can be derived using symmetries.}

Corresponding to a Dirac point at $(\bk_{\rm D},E_D)=(\bK,E_D)$, is a two-dimensional eigenspace
 $L^2_{\bK}-$ eigenspace of~$H^0$:
    \[  {\rm nullspace}_{L^2_{\bK}}\left(H^0-E_D I\right)\ =\ \{\Phi_1(\bx),\Phi_2(\bx)\}.\]  
Using honeycomb symmetries, a basis can be chosen such that
 \begin{equation}
 \RRR[\Phi_1]=\tau \Phi_1, \quad \RRR[\Phi_2]=\overline\tau\Phi_2\quad   {\rm and}\quad 
 \Phi_2(\bx)=\overline{\Phi_1(-\bx)}.\label{tbt} 
 \end{equation}
Here, $1, \tau$ and $\overline\tau$ denote the cubic roots of unity.   We next record inner product relations, consequences of  \eqref{tbt},  which
   play an important role in the derivation of effective Dirac dynamics; see Section \ref{sec:Dirac} and \cite{FW:14}.

 \begin{proposition}\cite[Prop.\ 4.1]{FW:12} \label{ipPhi}
 \begin{align}
\left\langle \Phi_1,\nabla\Phi_1\right\rangle_{L^2_{\bK}} &=\ 
\left\langle \Phi_2,\nabla\Phi_2\right\rangle_{L^2_{\bK}} 
=\ \begin{pmatrix} 0\\ 0\end{pmatrix},\quad a=1,2 \label{1nab1}
\end{align}
and 
 \begin{equation}\label{vF}
\left\langle \Phi_1,-2i\nabla\Phi_2\right\rangle_{L^2_{\bK}} =\ v_{\rm D}\ \begin{pmatrix} 1\\ i\end{pmatrix} .
\end{equation}
The constant, $v_{\rm D}=\frac12(1, -i) \left\langle \Phi_1,-2i\nabla\Phi_2\right\rangle_{L^2_{\bK}}$, is known as the Dirac velocity  or Fermi velocity. By an appropriate choice of phase for $\Phi_1$,  $v_{\rm D}$ can be chosen to be non-negative
 and has been proved to be generically nonzero \cite{FLW-MAMS:17, FW:12}.
 \end{proposition}
 
Dirac points are robust in the following sense \cite{FW:12,LWZ:18}: the conical intersection of dispersion surfaces persists under  sufficiently small perturbations of $H^0$ which are $\Lambda-$ periodic and invariant under $\mathcal{P}\circ\mathcal{C}$. Under such perturbations, a Dirac cone may deform to an elliptical cone and the cone vertex 
may perturb away from a vertex of $\mathcal{B}$. On the other hand, a small perturbation which breaks either $\mathcal{P}$ or $\mathcal{C}$ symmetry leads to a \underline{local} gap opening, about $(\bK,E_D)$, {\it i.e.} for $\bk$ sufficiently near the vertices of $\mathcal{B}$.

\section{Effective dynamics for Dirac wave-packets}\label{eff-Dirac}

A natural class of initial conditions for \eqref{gen-model} are those whose Floquet-Bloch decomposition is concentrated in a small neighborhood 
 of a Dirac point. Indeed, this is the class of excitations for which the remarkable properties of graphene and its engineered analogues 
  have been widely explored, theoretically and experimentally \cite{RMP-Graphene:09}. 
A way to construct such data is through a slow and spatially decaying modulation of the degenerate subspace, associated with a Dirac point. The spectrally concentration of such functions about 
 the Dirac point is discussed in Section~\ref{sec:bl}.

We define a  {\it Dirac wave-packet}, associated with the Dirac point $(\bK,E_{\rm D})$ to be a two-scale function of the form:
\begin{align}\label{eq:slow_wp}
\psi^\varepsilon_{wp}(\bx) &= \varepsilon \sum\limits_{j=1,2} \alpha_{j,0} (\varepsilon \bx) \Phi_j (x) = \varepsilon {\alpha}_0(\varepsilon \bx)^{\top} \Phi(\bx;\bK) ,\quad \textrm{where}\\
\alpha_0(X) &= \begin{pmatrix}\alpha_{1,0}(X) \\ \alpha_{2,0}(X)\end{pmatrix},\quad {\Phi}(\bx;\bK) = \begin{pmatrix}\Phi_1(\bx) \\ \Phi_2(\bx)\end{pmatrix} .
\end{align}
Here, $\big\{\Phi_1(\bx),\Phi_2(\bx)\big\}$ denotes the distinguished basis associated with the Dirac point, introduced
 in Section \ref{hlp-dp} and $\alpha_0\in H^s(\mathbb{R}^2;\mathbb{C}^2)$ for $s\geq 0$. The parameter
$\varepsilon>0$ is taken to be small and 
the prefactor of $\varepsilon$ in \eqref{eq:slow_wp} ensures that $\|\psi^\varepsilon_{wp}\|_2\sim \|\alpha_0\|_2$ is of order $1$. 

A key part of our analysis is the observation that the evolution of  Dirac wave packet initial data \eqref{eq:slow_wp},
 under the Schr{\"o}dinger equation \eqref{eq:lsA} is well-approximated, on very long time scales, by the two-scale function $\varepsilon \alpha(\varepsilon \bx,\varepsilon t)^\top \Phi(\bx,\bk _D)$, where 
  the envelope functions $\alpha(X,T)=(\alpha_1,\alpha_2)^\top$ evolve according to an {\it effective (homogenized) magnetic Dirac Hamiltonian}  $\slashed{D}_{A}(T)$ with magnetic potential $A(T)$:
   \begin{equation} 
\slashed{D}_A(T) \equiv v_{\rm D}\Big[ \left(\frac{1}{i} \partial_{X_1}+A_1(T)\right)\sigma_1 -  \left(\frac{1}{i}\ \partial_{X_2} + A_2(T)\right)\sigma_2 \Big] \, .
\label{Dirac-op}\end{equation}
Here, $\sigma_1$ and $\sigma_2$ denote standard Pauli matrices; see \eqref{pauli}.

\begin{remark}\label{def-Dirac}
 Effective magnetic Dirac operators have been derived to explain phenomena in other settings; see, for example,
the recent work on strained photonic crystals and Landau levels, \cite{GRW:21}, and references cited therein. 
\end{remark}

We shall write
\[   {\alpha}(T) =  U_{\rm Dir}(T)\alpha_0\quad {\rm or}\quad {\alpha}(T,X) =  U_{\rm Dir}[\alpha_0](X,T)\]
for the solution of the initial value problem (IVP)
     \begin{align}  
      i \partial_T \alpha(T,X) &= \slashed{D}_{A}(T) \alpha(T,X), \qquad \alpha(0,X) = \alpha_0(X)\in H^s(\mathbb{R}^2;\mathbb{C}^2),\quad  s\ge0. 
\label{eq:diracA}
\end{align} 
Since $\slashed{D}_{A}(T)$ is self-adjoint, the  Dirac evolution \eqref{eq:diracA} is unitary in $L^2(\R^2_X; \C ^2)$. Furthermore, since $\slashed{D}_{A}(T)$ commutes with spatial translations we have, for any $s\ge0$ it is also unitary in $H^2(\R^2 ;\C^2)$:
\begin{equation}
 \|\alpha(T)\|_{H^s(\R^2 ;\C ^2)}\ =\  \|\alpha_0\|_{H^s(\R^2 ;\C ^2)},\quad \textrm{for all $T$}.\label{unitarity}\end{equation}

\begin{theorem}[Effective magnetic Dirac dynamics]
\label{thm:valid}
There exists $\varepsilon_0>0$ such that for all $0<\varepsilon<\varepsilon_0$, the following holds:
 Consider \eqref{eq:lsA}, the Schr{\"o}dinger equation with time periodic Hamiltonian $H^\varepsilon(t)$, and initial data, $\psi^\varepsilon_{wp}$ of Dirac wave packet type
 \eqref{eq:slow_wp} with $\alpha_0 \in H^4(\mathbb{R}^2;\mathbb{C}^2)$. 
Fix constants $T_0>0$ and  $0<\rho < 1$. 

{\rev  Then, there exists a constant $C$, which depends on $\rho$, $T_0$, and $\|{\alpha} _0\|_{H^4(\R ^2;\C ^2)}$,} such that
\begin{align}\label{eq:psi0_validity}
 \Big\|U^\varepsilon[\psi^\varepsilon_{wp}](t,\bx) - \varepsilon U_{\rm Dir}[\alpha_0](\varepsilon \bx,\varepsilon t)^\top \Phi(\bx;\bk _D) e^{-iE_Dt} \Big\|_{L^2(\mathbb{R}_x^2)}  &\le C \varepsilon^{\rho}  \, , \qquad 0\le t\le  T_0\ \varepsilon^{-(2-\rho)} \,
\end{align}
where $\psi_0\mapsto U^\varepsilon(t)\psi_0$ denotes the solution of Schr{\"o}dinger equation \eqref{eq:lsA} with $\psi(0,x) = \psi_0$.
\end{theorem}
To prove Theorem \ref{thm:valid} we first derive an effective (homogenized) Dirac equation, via a formal multiple scale expansion, which we expect captures the dynamics on the desired long time-scale, and we then estimate the error in this approximation. 
The details are presented in Section \ref{sec:Dirac}. 
 
 Since the effective dynamics given in Theorem \ref{thm:valid} are valid, with small error, on a time scale much larger than
  the period of temporal forcing $\sim\varepsilon^{-1}$, 
  we can approximate the monodromy operator, $M^\varepsilon=U^\varepsilon(\varepsilon^{-1}\Tper)$ for the Schr{\"o}dinger evolution  \eqref{eq:lsA} applied to Dirac wave packets
   using the monodromy operator, $M_{\rm Dir}=U_{\rm Dir}(\Tper)$ of the effective Dirac dynamics.

\begin{corollary}[$M_{\rm Dir}$ as an approximation of $M^\varepsilon$]\label{cor:mono_approx}
Assume  $\alpha_0 \in H^4(\mathbb{R}^2;\mathbb{C}^2)$.  Then for $0<\varepsilon<\varepsilon_0$ sufficiently small
 and wave-packet data $\psi^\varepsilon_{wp}$:
\begin{align*}
 \Big\|(M^\varepsilon \psi^\varepsilon_{wp})(\bx) - \left(\tMD\psi^\varepsilon_{wp}\right)(\bx) \Big\|_{L^2(\mathbb{R}^2_x)} &\le  
   C \varepsilon\ , 
\end{align*}
where 
\begin{equation} \left(\tMD\psi^\varepsilon_{wp}\right)(\bx) \equiv \varepsilon (M_{\rm Dir}\alpha_0)(\varepsilon \bx)^\top \Phi(\bx) 
e^{-iE_D (\Tper/\varepsilon)}. \label{tMD-def}
\end{equation}
The constant, $C$,  depends on the $H^4(\R ^2 ; \C ^2)$ norm of $\alpha_0$ and is independent of $\varepsilon$.
\end{corollary}
\begin{remark}
In both Theorem \ref{thm:valid} and Corollary \ref{cor:mono_approx}, it is possible to relax the assumption $\alpha_0\in H^4$. We do not pursue this here.
\end{remark}
\subsection{Floquet-multiplier / quasi-energy gap for effective Dirac dynamics}\label{D-gap}

In this section we show that on a subspace of band-limited $L^2(\mathbb{R}^2;\mathbb{C}^2)$  functions,
  the monodromy operator for the effective dynamics, $M_{\rm Dir}$, has a Floquet-multiplier gap on the unit circle $S^1$. 
Our main result,  Theorem \ref{thm:quasigap_general},  shows that this property extends to an {\it effective gap} for the 
 monodromy operator $M^\varepsilon$ associated with the Schr{\"o}dinger evolution \eqref{eq:lsA}.

To facilitate explicit computations, we work with the specific periodic forcing \cite{ablowitz2017tight, Rechtsman-etal:13}
\begin{equation}\label{eq:A_circ}
A(T) = R (\cos (\omega T), \sin(\omega T) ) .
\end{equation}
 First note that $U_{\rm Dir}(t)$, the Dirac flow of \eqref{eq:diracA}, is unitary on $L^2(\mathbb{R}^2)$. Therefore the spectrum of $M_{\rm Dir}$ lies on the unit circle $S^1$.   Furthermore, $\slashed{D}_{A}(T)$ is invariant under arbitrary spatial translations, so we  can solve \eqref{eq:diracA} via the Fourier transform. Let $\hat\alpha_j(T,\bxi) $ denote the Fourier transform on $\mathbb{R}^2$ of $\alpha_j(T,X)$.  Then, for each $\bxi=(\xi_1,\xi_2)\in \mathbb{R}^2$, $\hat\alpha(T,\bxi)$
  satisfies the system of periodic ODEs:
 \[ i\partial_T \hat\alpha(T,\bxi)  = \hat{\slashed{D}}(T;\bxi)\hat\alpha(T,\bxi) ,\quad {\rm where}\]
\begin{align}\label{eq:dirac_xi}
   \hat{\slashed{D}}(T;\bxi) &= v_{\rm D}\left( \left(\xi_1+A_1(T)\right)\sigma_1- \left(\xi_2+A_2(T)\right)\sigma_2
  \right)\nonumber\\ 
   &= \   v_{\rm D} \begin{pmatrix}
         0& \xi_1 + i\xi_2   +  R  e^{i\omega T} \\
         \xi_1 - i\xi_2 + R   e^{-i\omega T} & 0
    \end{pmatrix}
\end{align} 
Let $\hat{U}_{\rm Dir}(T,\bxi):\mathbb{C}^2\to\mathbb{C}^2$  denote
the fundamental matrix for \eqref{eq:dirac_xi} with $\hat{U}_{\rm Dir}(0,\bxi)=I_{2\times2}$. Then, 
\begin{equation}
(U_{\rm Dir}\alpha_0)(X,T)= \frac{1}{(2\pi)^2}\int_{\mathbb{R}^2} e^{i\bxi\cdot X}\hat{U}_{\rm Dir}(T,\bxi)\hat\alpha_0(\bxi) d\bxi
 \label{UDirac}\end{equation}
 and hence the monodromy operator
 \begin{equation}
(M_{\rm Dir}\alpha_0)(X)= \frac{1}{(2\pi)^2}\int_{\mathbb{R}^2} e^{i\bxi\cdot X}\hat{U}_{\rm Dir}(\Tper,\bxi)\hat\alpha_0(\bxi) d\bxi
 \label{MDirac}\end{equation} 
Since ${\rm trace}(\hat{\slashed{D}}(T;\bxi))=0$,  then $\hat{U}_{\rm Dir}(\Tper,\bxi)$ has two eigenvalues (Floquet multipliers), 
 $\lambda_+(\bxi)$ and  $\lambda_-(\bxi)$, which 
 satisfy $\lambda_+(\bxi)\lambda_-(\bxi)=1$ \cite{coddington1955theory}. By unitarity of $\hat{U}_{\rm Dir}(T,\bxi)$, they lie on the unit circle and satisfy: $\lambda_-(\bxi)=\overline{\lambda_+(\bxi)}$. We set 
\begin{equation}\label{eq:floq_mult} \lambda_+(\bxi)=\exp(i\mu(\bxi)\Tper)\quad {\rm and}\quad \lambda_-(\bxi)=\exp(-i\mu(\bxi)\Tper) .
\end{equation}
Since $\mu(\bxi)\Tper$ is defined modulo $2\pi$ we take $\mu(\bxi)\Tper\in[0,2\pi)$.  Choose an orthonormal set of eigenvectors 
$\{v_+(\bxi),v_-(\bxi)\}$ of $\hat{U}_{\rm Dir}(\Tper,\bxi)$:
\begin{equation}  \hat{U}_{\rm Dir}(\Tper,\bxi) v_\pm(\bxi)= \exp(\pm i\mu(\bxi)\Tper) v_{\pm}(\bxi)\ .\label{UDhat-evp}\end{equation}
It follows that 
\begin{equation} \hat{U}_{\rm Dir}(\Tper,\bxi)\hat\alpha_0(\bxi)= \sum_{r=\pm}\left\langle v_r(\bxi),\hat\alpha_0(\bxi)\right\rangle_{\C^2} e^{i r\mu(\bxi)\Tper}v_r(\bxi).
\label{UDhat}
\end{equation}

Since the effective Dirac equation \eqref{eq:diracA}  is spatially translation invariant, the unitary evolution $U_{\rm Dir}(T)$ 
acts invariantly on subspaces of compact Fourier support. For
 \begin{equation}\label{eq:MDd0_def}
 \textrm{
 $M_{{\rm D},d_0}=M_{\rm Dir}$, defined by \eqref{eq:diracA}, acting in the invariant subspace $\chi(|\nabla|\le d_0)L^2(\mathbb{R}^2)$,}
 \end{equation}
 the Hilbert space of $L^2$ functions whose Fourier transform vanishes
   for $|\xi|>d_0$; see Section \ref{sec:bl}. We next prove the following spectral gap result for  $M_{{\rm Dir},d_0}$; See upper panel in Figure \ref{fig:threeSpectra}.

\begin{proposition} \label{prop:Dirac_gap}
There exist constants $d_0>0$ and $\tilde{g}=\tilde{g}(d_0)$, both depending on the forcing~$A(T)$,
 such that  $M_{{\rm Dir},d_0}$ has a spectral gap on the unit circle;
$$ {\rm spec} (M_{{\rm D},d_0}) \cap  \{e^{iy} : |y|\le\tilde{g}\Tper\} =\ \emptyset\ .
$$ 
\end{proposition}

\begin{remark}\label{rem:BL}  
The band-limiting condition, $d_0>0$, of  Proposition \ref{prop:Dirac_gap} is necessary to ensure a gap in the spectrum of $M_{\rm Dir}$. Indeed, in Section \ref{sec:wkb} we prove, using WKB asymptotics, that the
 monodromy operator of $\hat{\slashed{D}}(T;\bxi)$ is well-approximated by $ \exp\left(i\left(v_{\rm D}( \sigma_1 \xi_1 -\sigma_2 \xi_2 \right)\Tper\right)$, an effective operator, whose eigenvalues $\exp(\pm i v_{\rm D}|\bxi|\Tper)$ and  that  
   \begin{equation}\label{eq:Diracspec_infty}  \bigcup\limits_{|\bxi|<d_0}\{\exp\left(\pm i\mu(\bxi)\Tper\right)\}\xrightarrow[]{d_0\to \infty} S^1 \, .
\end{equation}
It follows that ${\rm spec}(M_{\rm Dir})=S^1$; there are no gaps in ${\rm spec}(M_{\rm Dir})$.  That  \eqref{eq:Diracspec_infty} occurs, has previously noted in numerical simulations, see e.g., \cite{gu2011floquet,ibarra2019dynamical, kunold2020floquet}.

From a modeling perspective, the regime  of arbitrarily high momenta $\xi$ in $\alpha_0$ is outside the regime of validity of the effective  Dirac equation \eqref{eq:diracA}, which models the dynamics of {\em slowly}-varying envelope modulations of Bloch modes; see \eqref{eq:slow_wp}. 
\end{remark} 

 \begin{remark} Our main result, Theorem \ref{thm:quasigap_general}, applies to {\em all} periodic forcing functions, $A(T)$, for which $M_{\rm Dirac, d_0}$ has a spectral gap for some $d_0>0$. 
The particular choice \eqref{eq:A_circ} enables very explicit calculations.  By continuity arguments, small (time-periodic) perturbations of \eqref{eq:A_circ} will also have such a spectral gap. 
 \end{remark}

\begin{proof}[Proof of Proposition \ref{prop:Dirac_gap}]
The proof has two steps. We first show that for $\xi=0$, we have $\mu(0)\neq 0$ and therefore 
 the Floquet exponents are distinct. Then, by continuity,  for small $|\xi|$,  $|\mu(\xi)|$ is bounded away from zero. 

Consider first $\bxi=0$. Then \eqref{eq:dirac_xi} reduces to
\begin{equation}\label{eq:dirac_xi0}
     i\partial_T \hat\alpha(0,T) = \left( \begin{array}{ll}
         0&   v_{\rm D} R\ e^{i\omega T} \\
         v_{\rm D} R\ e^{-i \omega T} & 0
    \end{array} \right)\hat\alpha(0,T) =  \hat{\slashed{D}}(0,T)\hat\alpha(0,T) \ .
    \end{equation}
Defining
 \begin{equation}\label{eq:tildeadef}
a _1(T) \equiv \hat\alpha_1(T;0)e^{\frac{-i\omega}{2}T} \, , \qquad a _2(T) \equiv \hat\alpha_2(T;0)e^{\frac{i\omega}{2}T} \, ,
\end{equation}
we obtain the constant coefficient system
\begin{equation}\label{eq:tildea_ode}
i\partial_T
a (T) = \left(\begin{array}{ll}
\frac{\omega}{2} & R v_{\rm D} \\
R v_{\rm D} & -\frac{\omega}{2}
\end{array} \right) a(T)  \equiv L\  a(T)  \, .
\end{equation}
Denote by $z_\pm(0)$ and $v_\pm(0)$ the eigenvalues and corresponding eigenvectors of $L$. One verifies:
\begin{align*}
z_{\pm}(0) &=\pm \frac12\sqrt{\omega^2 + 4R^2 v_{\rm D}^2} 
\end{align*}
Let $V$ denote the matrix whose columns are $v_+(0)$ and $v_-(0)$. Then, 
 \begin{equation*}
 \hat{U}_{\rm Dir}(T,0)  = \begin{pmatrix} e^{\frac{i}{2}\omega T} & 0 \\ 0 & e^{-\frac{i}{2}\omega T} \end{pmatrix}
  e^{-iLT}, \quad \textrm{where}\quad   e^{-iL T} = V \begin{pmatrix} e^{-iz_+(0)T} & 0 \\ 0 &  e^{-iz_-(0) T} \end{pmatrix} V^{-1}\
\end{equation*}
For $\bxi=0$, we then have Floquet solutions

\begin{align*}
v_+(0) e^{-iz_+(0)T} e^{\frac{i}{2} \omega T} &= v_+(0) e^{-i(z_+(0) - \frac{\omega}{2}) T}\times e^{i\omega T}\\
   v_- (0) e^{-iz_-(0)T} e^{-\frac{i}{2} \omega T} &= v_-(0) e^{i(z_+(0) - \frac{ \omega }{2})T}\times e^{-i\omega T} \, ,
  \end{align*}
  where we used the relation $z_+(0) = -z_-(0)$.  Therefore, we have  $\hat{U}_{\rm Dir}(\Tper,0)v_\pm(0) = \lambda_\pm(0)v_\pm$
where $\lambda_\pm(0)\equiv e^{\pm i \mu(0)}$ are the distinct (complex conjugate) Floquet multipliers, with corresponding Floquet exponents

    \begin{equation} \mu_{\pm}(0)= \pm \left(z_+(0) -\frac{\omega}{2}\right) =\pm\frac12 \left( \sqrt{\omega^2 +4\rho ^2 v_{\rm D}^2} - \omega \right) \, .
    \label{mult0}\end{equation}
  By continuity, for $|\bxi|\le d_0$, where $d_0$ is chosen sufficiently small, $|\mu_\pm(\bxi)|\ge c_0$, where $c_0>0$ depends on $d_0$.
  It follows that there exists $C_1>0$, depending on $c_0$ and $d_0$, such that 
  \[ \sup_{|\bxi|\le d_0}\ \Big\| \left( \hat{U}_{\rm Dir}(\Tper,\bxi) - \lambda I_{2\times2} \right)^{-1} \Big\|_{\mathbb{C}^2\times\mathbb{C}^2}\le C_1\]
  for all $\lambda$ varying in an open arc on the unit circle, which contains $1$. 
  Finally, since
  \[
(M_{{\rm Dir},d_0}-\lambda I_{2\times2})^{-1}\alpha_0 = \frac{1}{(2\pi)^2}\int_{|\bxi|\le d_0} e^{i\bxi\cdot X}\left( \hat{U}_{\rm Dir}(\Tper,\bxi) - \lambda I_{2\times2} \right)^{-1} \hat\alpha_0(\xi) d\bxi
\, , \]
 it follows that this open arc is in the resolvent set of $M_{{\rm Dir},d_0}$.

\end{proof}
\section{Band-limited Dirac wave-packets}\label{sec:bl}

In Theorem \ref{thm:valid} we proved that the evolution  of Dirac wave-packets, $\psi^\varepsilon=\varepsilon\alpha_0(\varepsilon \bx)^\top\Phi(\bx)$ under the Schr{\"o}dinger  equation, \eqref{eq:lsA},  
 is given by an effective magnetic Dirac equation. In Proposition \ref{prop:Dirac_gap} we showed that the effective 
 Dirac dynamics, when restricted to a suitable invariant subspace of band-limited functions, $\chi(|\nabla|\le d_0)L^2(\R^2)$, 
    has a Floquet-multiplier (and therefore) quasi-energy gap.   This motivates the following:

 \begin{definition}[Band-limited Dirac wave-packets]\label{BLDpkt}
Fix $\varepsilon, d_0>0$ and let  $\bk _D\in\{\bK,\bK^\prime\}$. 
We say that {\rev $\psi\in {\rm BL}(d_0,\varepsilon)={\rm BL}(d_0, \varepsilon; k_D)\subset~ L^2(\mathbb{R}^2)$,} the subspace of
  {\it band-limited Dirac wave packets} with parameters $d_0$ and $\varepsilon$,
 if there exists $\alpha_\varepsilon\in L^2(\mathbb{R})$ with $ {\rm supp}\left(\hat\alpha_\varepsilon\right)\subseteq \{\bxi: |\bxi|\le d_0\} $ such that
   \[ \psi(\bx)= \alpha_\varepsilon(\varepsilon \bx)^{\top} {\Phi}(\bx;\bk _D).\]
Here, $\hat{\alpha}=\mathcal{F}(\alpha)$ is the Fourier transform of $\alpha$.
\footnote{Elsewhere in this article we use the scaling $\psi(\bx)= \varepsilon\alpha(\varepsilon \bx)^{\top} {\Phi}(\bx;\bk _D)$ to guarantee that $\|\psi\|_2 \approx \|\alpha\|_2$. Such $\psi$ satisfy the requirements  of ${\rm BL}(d_0,\varepsilon)$, with $\alpha$ replaced by $\varepsilon\alpha$.}  
 \end{definition}

 Band-limited wave-packets are a good physical model of the types of excitations considered  by  physicists, when exploring 
 phenomena related to  Dirac points of the unperturbed structure. A mathematically more intrinsic notion would be 
states in $L^2(\R^2)$, comprised only of Floquet-Bloch modes of $H^0 = -\Delta +V$ with energies within $\varepsilon$ distance from $E_D$. The next result 
shows that these states are Dirac wave-packets, up to a high-order correction. It is convenient to require the following property of the bulk  (unforced) Hamiltonian, $H^0$; see, for example, \cite{FLW-2d_edge:16,drouot2020edge} where a no-fold condition plays a role in the construction of edge states.
\begin{definition}[No-fold condition]\label{def:nofold}
We say that the Hamiltonian $H^0$ satisfies the no-fold condition if there exist constants $\delta_0, \delta_1 >0$, such that for all $0<\delta<\delta_1$ the following holds: if $\bk\in \mathcal{B}$ is such $|\bk - \bk_{\rm D}|>\delta$ {\rev for both $\bk_{\rm D}\in\{\bK, \bK^\prime\}$}, then ${\rm Proj}_{L^2_{\bk}}(|H^0 - E_D|<\delta_0)=0 \, .$
\end{definition}
The no-fold condition asserts that the bands which touch conically at a $(E_D,\bk_{\rm D})$ (and therefore all other high symmetry quasi-momenta at energy $E_D$) do not attain the energy $E_D$, outside a sufficiently small neighborhood of the high-symmetry quasi-momenta. Physically, this means that the bulk structure is semi-metallic at $(E_D,\bk_{\rm D})$. Although, in general, the no-fold condition may not hold, it has been proved to hold for graphene-like potentials in the strong binding regime; see~\cite{FLW-CPAM:17}. Furthermore, in applications bulk structures can be engineered to satisfy this condition. For an example, see \cite{GRW:21}.

 \begin{proposition}\label{prop:wp_dk}  Suppose the bulk Hamiltonian, $H^0$, satisfies the no-fold condition.
 \footnote{If $H^0$ does not satisfy the no-fold condition, then the conclusion
  of Proposition \ref{prop:wp_dk}  holds with\\ ${\rm Proj}_{L^2(\mathbb{R}^2)}(|H^0-E_D|<\varepsilon)$
   replaced by  $\int^\oplus_{|\bk-\bk_{\rm D}|<\delta}{\rm Proj}_{L^2_\bk}(|H^0-E_D|<\varepsilon)$
  with $\delta$ sufficiently small. }
   There exists $\varepsilon_0>0$ such that for all $0<\varepsilon<\varepsilon_0$, the following holds: for every $f\in L^2(\mathbb{R}^2)$ there are band-limited Dirac {\rev wave-packets, $u_\varepsilon^{\bK}[f]\in {\rm BL}(d_0,\varepsilon;\bK)$ and $u_\varepsilon^{ \bK^\prime}[f]\in {\rm BL}(d_0,\varepsilon;\bK^\prime)$,} such that
\begin{equation}
 {\rm Proj}_{L^2(\mathbb{R}^2)}(|H^0-E_D|<\varepsilon) f =  {\rev u_\varepsilon^{\bK}[f] +u_\varepsilon^{\bK^\prime}[f]  }+ \mathcal{O}\left(\varepsilon^3\|f\|_{L^2(\mathbb{R}^2)}\right) \, .
 \label{wp_dk1}
 \end{equation}
 Conversely, let  {\rev $u\in {\rm BL}(d_0,\varepsilon;\bk _{\rm D})$, where  $\bk_{\rm D}\in\{\bK,\bK^\prime\}$.} Then, 
\begin{equation}
\left\| {\rm Proj}_{L^2(\mathbb{R}^2)}(|H^0-E_D|>\varepsilon) u\right\|_{L^2(\mathbb{R}^2)} \le C\varepsilon^3 \|u\|_{L^2(\mathbb{R}^2)}  . 
\label{wp_dk2}
 \end{equation}
 \end{proposition}

In order to transfer the spectral gap information for the effective evolution (Section \ref{D-gap}) to the Schr{\"o}dinger evolution, \eqref{eq:lsA},  for which ${\rm BL}(d_0,\varepsilon)$ is not invariant, we introduce a decomposition of $L^2(\R^2)$ into
     a direct sum of ${\rm BL}(d_0,\varepsilon)$ and its orthogonal complement.

\begin{proposition}\label{BLDpkt-closed} 
For  any fixed $\varepsilon>0$ which is sufficiently small, 
$L^2(\mathbb{R}^2)$ has the orthogonal decomposition
\[ L^2(\mathbb{R}^2) = {\rm BL}(d_0,\varepsilon) \oplus {\rm BL}(d_0,\varepsilon)^\perp.\]
\end{proposition}

\begin{proof}[Proof of Proposition \ref{BLDpkt-closed}]\ 
It suffices to prove that ${\rm BL}(d_0,\varepsilon) $ is a closed subspace of $L^2(\mathbb{R}^2)$. Let ${\alpha}_n^{\top}(\varepsilon \cdot){\Phi} \to u$ in $ L^2(\mathbb{R}^2)$.
We prove that  $\{{\alpha}_n\}$ is a Cauchy sequence in $L^2(\mathbb{R}^2;\mathbb{C}^2)$.  This is a consequence of the following averaging lemma, which we prove in Section \ref{sec:pf_homog}.

\begin{lemma}[Averaging Lemma]\label{lem:homog}
Let $q \in L^2(\mathbb{R}^2)$ such that ${\rm supp}\, \hat{q}(\xi)\subseteq \{\bxi: |\bxi|\le d\}$ for some $d>0 $, and let $p\in L^2(\Omega)$ be $\Lambda$-periodic. Then, there exists $\varepsilon_0$ which depends $d$, such that for andy fixed  $0<\varepsilon<\varepsilon_0$, 
\begin{equation}\label{eq:homog}
\int\limits_{\mathbb{R}^2} p(\bx)q(\varepsilon \bx) \, d\bx = \varepsilon^{-2} \left( \int\limits_{\Omega} p(\bx)\, d\bx \right) \cdot \left( \int\limits_{\mathbb{R}^2} q(X) \, dX \right) \, . 
\end{equation}
\end{lemma}
%
\noindent
Now apply 
 Lemma \ref{lem:homog} with
 \[ q_{j,l}(X)= \overline{(\alpha_{j,n} - \alpha_{j,m})(X)} (\alpha_{l,n} - \alpha_{l,m})(X)\quad {\rm and}\quad p_{j,l}(\bx)= \overline{\Phi_j(\bx)}\Phi_l(\bx),\qquad j, l = 1,2 \, .\] 
  Since the set $\{\Phi_1,\Phi_2\}$ is orthonormal in $L^2(\Omega)$,  for all $\varepsilon>0$ sufficiently small:
  \begin{align*}
  \int_{\mathbb{R}^2}\big| \left(\alpha_n (\varepsilon \bx) - \alpha_m (\varepsilon \bx)\right)^{\top}\Phi(\bx)\big|^2\ d\bx &= \varepsilon^{-2} \sum\limits_{j,l=1,2} \left(\int\limits_{\Omega} p_{j,l}(\bx) \, d\bx\right) \cdot \left( \int\limits_{\mathbb{R}^2} q_{j,l}(X) \, dX\right) \\
  &= \varepsilon^{-2} \sum\limits_{j,l=1,2} \delta _{j,l} \cdot \left( \int\limits_{\mathbb{R}^2} q_{j,l}(X) \, dX\right) \\
  &=\varepsilon^{-2}\|\alpha_n - \alpha_m \|_2^2 . 
   \end{align*}
Since the left hand side tends to zero as $n$ and $m$ tend to infinity, so does the right hand side. Therefore, 
 $\{\alpha_n\}$ is Cauchy and converges in $L^2(\R^2 ; \C ^2)$ to some $\alpha_\star$, and therefore 
 the sequence $\{ \varepsilon \alpha_n^{\top}(\varepsilon \cdot) \Phi\}$ converges to
  $\varepsilon \alpha_\star^{\top}(\varepsilon \cdot) \Phi$ . 
  Finally, we claim that ${\rm supp}(\hat\alpha_\star)\subset \{|\bxi|\le d_0\}$.  Indeed, since   $\{\alpha_n\}$ converges in $L^2 $,  by the Plancherel identity, so does $\{\hat\alpha_n\}$ and hence
  $\hat\alpha_n(\bxi)$ converges almost everywhere, {\rev up to a subsequence.} Furthermore, for all $n$ we have $\hat\alpha_n(\bxi)\equiv0$ for $|\xi|>d_0$, so  we conclude $\hat\alpha_\star(\bxi)\equiv0$ for $|\bxi|>d_0$. This concludes the proof of Proposition \ref{BLDpkt-closed}
  \end{proof}

 \section{Main result; effective gap}\label{sec:mainres}

\begin{figure}[h]
\centering
{\includegraphics[scale=1]{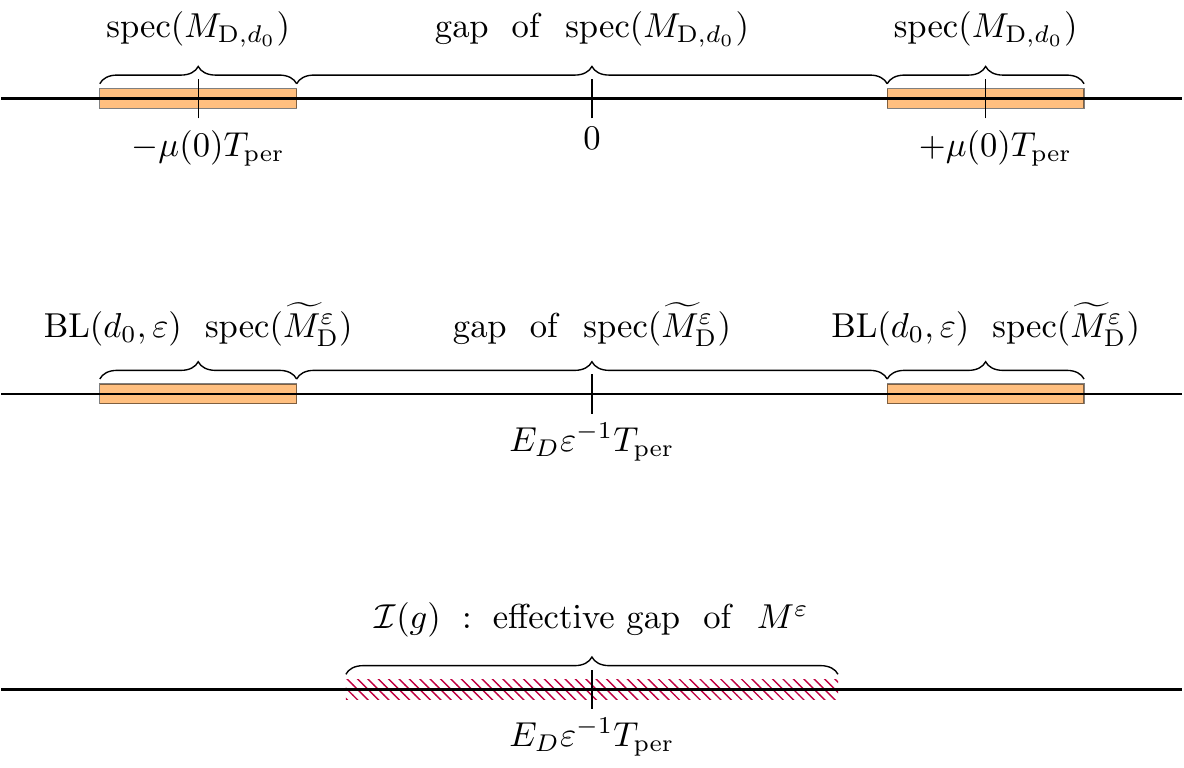}}
\caption{
Main results: Floquet spectra and gaps for the operators $M_{D,d_0}$ (top, see \eqref{eq:MDd0_def}), $M_D^{\varepsilon}$ (middle, see \eqref{tMD-def}), compared to the effective gap of the monodromy $M^{\varepsilon}$ of the Schr{\"o}dinger dynamics \eqref{eq:lsA} (bottom), see more details in Fig.\ \ref{fig:detail}. }
\label{fig:threeSpectra}
\end{figure}

By Proposition \ref{prop:Dirac_gap} the monodromy operator, $M_{{\rm Dir},d_0}$, associated with the effective
 Dirac evolution acting in $\chi(|\nabla|\le d_0)L^2$, has 
 a spectral gap on the unit circle, the arc $\{e^{iy}: |y|<\tilde{g}\Tper\}$. For any $g$ such that $0<g<\tilde{g}$ consider, via Proposition \ref{cor:mono_approx} for  approximating $M^\varepsilon$, the closed arc on the unit circle:
   \begin{equation} 
  \{e^{-i\nu} ~~|~~ \nu \in \mathcal{I} \} \, , \qquad \mathcal{I} = \mathcal{I}(g) \equiv [(\varepsilon^{-1}E_{\rm D} -g)\Tper , (\varepsilon^{-1}E_{\rm D} +g)\Tper ] \, .
\label{sub-arc}\end{equation}
We expect this arc to be filled with spectrum of $M^\varepsilon$; see Remark \ref{rem:BL}. Let $\Pi^\varepsilon$ denote the spectral measure associated with the unitary operator $M^\varepsilon$
\begin{equation}\label{eq:spectral_measure}
 M^\varepsilon = \int_{z\in {\rm spec}(M^\varepsilon)\subset S^1} z\ d\Pi^\varepsilon(z);
 \end{equation} 
see Appendix \ref{ap:spectral} and, for example, \cite{hall2013quantum, taylor2013partial} . The nature of
$\Pi^\varepsilon$ is largely an open problem but it is expected that the arc \eqref{sub-arc} is contained in its support. 
Our main result concerns the nature of any $L^2$ states formed 
via superposition of modes with quasi-energies in $\mathcal{I}$. No assumptions are made on the projection valued measure $z\mapsto\Pi^\varepsilon(z)$.

Let $\mathcal{I}(\tilde{g})$ denote the quasi-energy gap of $M_{{\rm Dir},d_0}$:
 \[ \mathcal{I}(\tilde{g}) \equiv \left(E_D\varepsilon^{-1}\ +\ [ -\tilde{g},\tilde{g} ]\right)\Tper;\quad \textrm{see \eqref{sub-arc}.}\] 

\begin{theorem}\label{thm:quasigap_general}
Consider the parametrically forced Schr{\"o}dinger equation \eqref{eq:lsA} with periodic forcing $A(T)$ given by~\eqref{eq:A_circ}. 
Fix $0<g<\tilde{g}$, so that  $\mathcal{I}(g)\subset \mathcal{I}(\tilde{g})$.
There exist constants $\varepsilon_0$, $C_1$, and $C_2$, such that if $|\varepsilon|<\varepsilon_0$ and $\mathcal{I}^\prime\subset\mathcal{I}(g) $ is an interval such 
that $|\mathcal{I}^\prime|\le C_1(\tilde{g}-g)$,  and if $u\in L^2(\mathbb{R}^2)$ is such that
\[ \Pi^{\varepsilon}(\{e^{-i\nu}  ~|~\nu \in \mathcal{I}'\})u  = u, \quad u\ne0, \]
then
 $$\|({\rm I}-{\rm Proj}_{{\rm BL}(d_0,\varepsilon)})u\|_{L^2(\mathbb{R}^2)} \geq C_2(\tilde{g}-g)\|u\|_{L^2(\mathbb{R}^2)} . $$
\end{theorem} 
\begin{figure}[h]
\centering
{\includegraphics[scale=1]{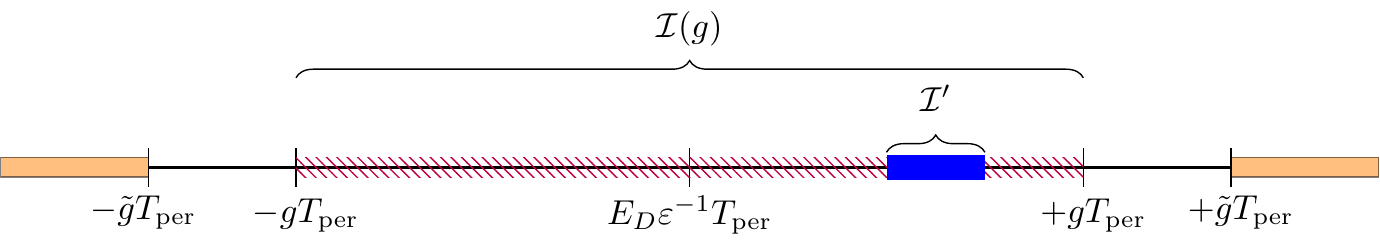}}
\caption{Effective gap result, Theorem \ref{thm:quasigap_general}, in terms of Floquet exponents.} 
\label{fig:detail}
\end{figure}
\begin{remark}\label{no-spec}
The case of a proper spectral gap, i.e., where the range of $\Pi^{\varepsilon}(\{e^{-i\nu} ~~|~~\nu \in \mathcal{I} \}$ is $\{0\}$, is certainly covered by Theorem \ref{thm:quasigap_general}. Indeed, if there is an arc in $S^1$ with no spectrum, then by our definitions it is also an ``effective gap'' .

\end{remark}

\begin{remark}\label{flat} As conjectured earlier, we expect the spectrum of the monodromy operator to be all of $S^1$. Concerning the spectral measure, we do not expect it to have a pure point component.  Indeed, if $z\in S^1$ is an $L^2(\mathbb{R}^2)$ eigenvalue of $M^{\varepsilon}$ if and only if $z$ is an $L^2_k$ eigenvalue of $M^{\varepsilon}$ for a set of quasi-momentum $\bk\in \mathcal{B}$ with positive Lebesgue measure \cite[Theorem XIII.85]{RS4}. This is possible, for example, if the Hamiltonian or Floquet Hamiltonian \eqref{eq:Floq_hamil} to have ``flat quasi-energy bands'',  which is considered highly unlikely for \eqref{eq:lsA}. An alternative, more physical argument, is to observe that if indeed ${\rm spec}(M^{\varepsilon}) = S^1$, then point spectrum would correspond to embedded eigenvalues in the continuous spectrum. For generic potentials $V$ and forcings $A$ one would expect such states to couple to the continuum be, at best, long lived (metastable) scattering resonance modes of finite lifetime \cite{yajima1982resonances}.
\end{remark}

{\rev Theorem 5.1 gives quantitative information on the effective gap; any state supported in the spectral subspace of $M^\varepsilon$ associated with interval $\mathcal{I}$ (see (5.1)) cannot be arbitrarily concentrated in ${\rm BL}(d_0, \varepsilon)$. In particular, such a state must contain a non-negligible projection onto ${\rm BL}(d_0, \varepsilon)^{\perp}$.}
While the ${\rm BL}(d_0,\varepsilon)$ condition may seem highly restrictive, the fact that Theorem  \ref{thm:quasigap_general} applies even for modes with \textit{some} ${\rm BL}(d_0,\varepsilon)^{\perp}$ parts mean that the mode in question can contain:
\begin{enumerate}
\item Modulations of higher-order Bloch modes at $k_{D}$, i.e., $\Phi_3, \Phi_4$ etc.
\item Components of quasi-momentum $k$ far away from $\bk _D$.
\item Modulations $\vec{\beta}(\varepsilon \bx)^{\top}\Phi(x)$ with $\vec{\beta}\in L^2(\mathbb{R}^2;\mathbb{C}^2)$ such that $\beta$ contains high Fourier-frequencies, i.e., $\vec{\beta }\in \chi (|\nabla |>d_0 )L^2(\mathbb{R}^2;\mathbb{C}^2) $.
\end{enumerate}

\section{Proof of Theorem \ref{thm:quasigap_general}}\label{sec:mainpf}

Fix $0<g<\tilde{g}$, $\mathcal{I}(g)\subset \mathcal{I}(\tilde{g})$. Let  $\varepsilon_0$, $ C_1$, and $C_2$ be constants to be determined. Let $|\varepsilon|<\varepsilon_0$ and $\mathcal{I}^\prime\subset\mathcal{I}(g) $ be  such 
that $|\mathcal{I}^\prime|\le C_1(\tilde{g}-g)$. Finally, let  $u\in L^2(\mathbb{R}^2)$, $u\ne0$ and such that
$\Pi^{\varepsilon}(\{e^{-i\nu}  ~|~\nu \in \mathcal{I}'\})u  = u$.

By the orthogonal decomposition of $L^2(\R^2)$ in Proposition \ref{BLDpkt-closed}, we may decompose  $u$ as:\begin{align}
u(\bx) &= u_{wp}^\varepsilon(\bx) + r(\bx),\qquad \textrm{where} \label{u-expand}\\ 
u_{wp}^\varepsilon &\in {\rm BL}(d_0,\varepsilon),\quad   r\in {\rm BL}(d_0,\varepsilon)^\perp .   \label{r-est}
   \end{align}
%
%
%
{\em To prove Theorem \ref{thm:quasigap_general}, we need to bound $\|r\|_{L^2(\R^2)}$ from below.} The monodromy operator of the effective Dirac equation, $M_{{\rm Dir},d_0}$,
 gives rise to an approximation to the  monodromy operator for the Schr{\"o}dinger equations \eqref{eq:lsA}, $\tMD$ (defined in \eqref{tMD-def}). As defined, $\tMD$ only acts on the space ${\rm BL}(d_0,\varepsilon)$. However we may extend it to all $L^2(\R^2)$: 
\begin{lemma}\label{extend}
$\tMD$ extends to all $L^2(\R^2)={\rm BL}(d_0,\varepsilon)\oplus {\rm BL}(d_0,\varepsilon)^\perp$ as a unitary operator.
\end{lemma}
\begin{proof}
Recall that we have defined $\tMD$ on the closed subspace ${\rm BL}(d_0,\varepsilon)$. For $r\in{\rm BL}(d_0,\varepsilon)^\perp$ we define  $\tMD r=r$. By the orthogonal decomposition of $L^2(\R^2)$ (Proposition \ref{BLDpkt-closed}) and linearity, the extended operator
   $\tMD$ is unitary on all $L^2(\mathbb{R}^2)$. 
   \end{proof}
By the effective dynamics approximation result,  Corollary \ref{cor:mono_approx}, we have 
\begin{equation}
\|(M^{\varepsilon}-\tMD) u_{\rm wp}^{\varepsilon} \|_{L^2(\mathbb{R}^2)}  \le C \varepsilon,  \, .\label{est}\end{equation}
where $C$ is independent of $\varepsilon$.\footnote{
The upper bound in Corollary \ref{cor:mono_approx} generally depends on $\|\alpha_0\|_{H^4(\mathbb{R}^2)}$. However, for any $u\in \chi(|\nabla |< d_0)L^2(\mathbb{R}^2)$ with $\|u\|_{L^2(\mathbb{R}^2)}=1$, then $\|u\|_{H^4(\mathbb{R}^2)}$ are uniformly bounded from above, and therefore we omit the explicit dependence on this norm.}
Using that $u_{\rm wp}^{\varepsilon}=u-r$, $\tMD r=r$, and  \eqref{est},
 we have 
\begin{align}
\|(M^{\varepsilon}-\tMD)u \|_{L^2(\mathbb{R}^2)} &\leq \|(M^{\varepsilon}-\tMD)u_{\rm wp}^{\varepsilon} \|_{L^2(\mathbb{R}^2)} + \|(M^{\varepsilon}-\tMD)r \|_{L^2(\mathbb{R}^2)} \nonumber \\
&\le C\varepsilon   + 2\|r\|_{L^2(\mathbb{R}^2)} 
\label{MtMD}\end{align}
\noindent
We next obtain, for $\varepsilon$ sufficiently small,  a lower bound on 
$\|(M^{\varepsilon}-\tMD)u \|_{L^2(\mathbb{R}^2)}$. Two key steps are required for this lower bound: 
\begin{enumerate}
\item {\bf 	Proposition \ref{red-app-estate}:} Using that $u$ is spectrally concentrated on $\mathcal{I}'$, with respect the measure $\Pi^{\varepsilon}$, we  prove that 
for $\nu_0 \in \mathcal{I}^\prime$
\[ M^{\varepsilon}u = e^{-i\nu_0}u +\rho,\quad {\rm where}\quad \|\rho\|\leq |\mathcal{I}'|\cdot \|u\|_{L^2(\R ^2)}.\]
\item {\bf Lower bound on $\| (e^{-i\nu_0} - \tMD)u_{wp}^{\varepsilon} \|_{L^2(\mathbb{R}^2)}$ :}  This is proved 
using scale separation, the spectral measure of $M_{\rm Dir}$, and a multiscale averaging lemma (Lemma \ref{lem:homog}). 
\end{enumerate}
Since $u$ is an approximate eigenvector of $M$ (step 1.), the latter lower bound and \eqref{MtMD} together provide a lower bound on  $r\in {\rm BL}(d_0,\varepsilon)^{\perp}$, defined in \eqref{r-est}, which will imply
Theorem \ref{thm:quasigap_general}.
\begin{proposition}\label{red-app-estate}
Let $\mathcal{I}^\prime\subset\mathcal{I}(g)\subset\mathcal{I}(\tilde{g})$ as above. Choose $\nu_0$ to be the midpoint of $ \mathcal{I}'$.  Then, 
\[ M^\varepsilon u = e^{-i\nu_0}u+\rho,\quad {\rm where}\quad 
 \|\rho\|_{L^2(\mathbb{R}^2)} \le |\mathcal{I}^\prime|\cdot \|u\|_{L^2(\mathbb{R}^2)}.\] 
\end{proposition}
\begin{proof} Since $\Pi^{\varepsilon}(\{e^{-i\nu} ~|~\nu \in \mathcal{I}' \})u = u$, we have using the spectral representation of the unitary operator $M^\varepsilon$ (Appendix \ref{ap:spectral}):
\begin{align*}
M^{\varepsilon}u & =\int\limits_{\{z=e^{-i\nu} : \nu \in \mathcal{I}' \}} M^{\varepsilon}\, d\Pi^{\varepsilon}(z) u =\int\limits_{\mathcal{I}'} e^{-i\nu} \, d\Pi^{\varepsilon}(e^{-i\nu})u  \\
 &= \int\limits_{\mathcal{I}'} e^{-i\nu_0} \,d\Pi^{\varepsilon}(e^{-i\nu})u  + \int\limits_{\mathcal{I}'} (e^{-i\nu}-e^{-i\nu_0}) \,d\Pi^{\varepsilon}(e^{-i\nu})u \\
&= e^{-i\nu_0}u + \rho,\quad {\rm where}\quad \rho \equiv \int\limits_{\mathcal{I}'} (e^{-i\nu}-e^{-i\nu_0}) \,d\Pi^{\varepsilon}(e^{-i\nu})u \,  .\end{align*}
\noindent
Therefore,
\begin{align*}
\|\rho \|_{L^2(\mathbb{R}^2)}^2 &= \Big\| \int\limits_{\mathcal{I}'} \left(  e^{-i\nu} - e^{-i\nu_0} \right) \, d\Pi^{\varepsilon}(e^{-i\nu})u \Big\| _{L^2(\mathbb{R}^2)}^2 \\
&=  \int\limits_{\mathcal{I}'} \big| e^{-i\nu} - e^{-i\nu_0} \big|^2 \, \left\langle d\Pi^{\varepsilon}(e^{-i\nu})u,u\right\rangle 
    \le |\mathcal{I}'|^2 \cdot \|u \|_{L^2(\mathbb{R}^2)}^2\ .
\end{align*}
 This completes the proof of Proposition \ref{red-app-estate}.
\end{proof}

Proposition \ref{red-app-estate} implies
\begin{equation}
 \Big\| \left(M^\varepsilon-\tMD\right)u \Big\|_{L^2(\mathbb{R}^2)} \ge \Big\| \left(e^{-i\nu_0}-\tMD\right)u \Big\|_{L^2(\mathbb{R}^2)}  -  |\mathcal{I}'|\cdot \|u\|_{L^2(\mathbb{R}^2)} .\label{a1}
 \end{equation}
The upper bound  \eqref{MtMD}  and the lower bound \eqref{a1} gives:
\begin{align}
\Big\| \left(e^{-i\nu_0}-\tMD\right)u \Big\|_{L^2(\mathbb{R}^2)}  -  |\mathcal{I}'|\cdot \|u\|_{L^2(\mathbb{R}^2)} &\le 
 C\varepsilon   + 2\|r\|_{L^2(\mathbb{R}^2)}
\label{uplow}\end{align}

We now require a lower bound on $  \|( e^{-i\nu_0}I -\tMD )u \|_{L^2(\mathbb{R}^2)}$ for all $\varepsilon$ sufficiently small. 
 Since $u=u_{wp}^\varepsilon + r$ and since $\tMD$ is the identity on ${\rm BL}(d_0,\varepsilon)^{\perp}$ (Lemma \ref{extend}) we have, by the triangle inequality,
\begin{equation}\label{eq:tri_centered}
\Big\|\left( e^{-i\nu_0}I -\tMD\right)u \Big\|_{L^2(\mathbb{R}^2)}
\ge  \Big\|\left( e^{-i\nu_0}I -\tMD\right)u_{wp}^\varepsilon \Big\|_{L^2(\mathbb{R}^2)} - 
 \|r\|_{L^2(\mathbb{R}^2)} \, ,
\end{equation}
which together with \eqref{uplow} yields:
\begin{equation}
\Big\|\left( e^{-i\nu_0}I -\tMD\right)u_{wp}^\varepsilon \Big\|_{L^2(\mathbb{R}^2)} \le C\varepsilon + 3 \|r\|_{L^2(\mathbb{R}^2)}
 + |\mathcal{I}'|\cdot \|u\|_{L^2(\mathbb{R}^2)} .
\label{lb-2}
\end{equation}
Thus we now seek a lower bound for $ \|( e^{-i\nu_0}I -\tMD)u_{wp}^\varepsilon \|_{L^2(\mathbb{R}^2)}$. 

\begin{proposition}\label{lb*} For all $\varepsilon>0$ sufficiently small,
\begin{equation}\label{eq:gamma_lbd}
\left\| (e^{-i\nu_0} - \tMD)u_{wp}^{\varepsilon} \right\|_{L^2(\mathbb{R}^2)} \geq  \left(\frac12 \big|\mathcal{I}^\prime\big| + (\tilde{g}-g)\Tper \right)\|\alpha_0\|_{L^2(\mathbb{R}^2;\mathbb{C}^2)} \,   .
\end{equation}
\end{proposition}

Before proving Proposition \ref{lb*} we use it to conclude the proof of Theorem \ref{thm:quasigap_general}. By
 \eqref{eq:gamma_lbd} and \eqref{lb-2} we have 
 \begin{align*}
   \left(\frac12 \big|\mathcal{I}^\prime\big| + (\tilde{g}-g)\Tper \right)\|\alpha_0\|_{L^2(\mathbb{R}^2;\mathbb{C}^2)}
& \le C\varepsilon + 3 \|r\|_{L^2(\mathbb{R}^2)}
 + |\mathcal{I}'|\cdot \|u\|_{L^2(\mathbb{R}^2)}\\
 & \le C\varepsilon + 3 \|r\|_{L^2(\mathbb{R}^2)}
 + |\mathcal{I}'|\cdot \left( \|u_{wp}^\varepsilon\|\ +\ \|r\| \right)\\
 & = C\varepsilon + 3 \|r\|_{L^2(\mathbb{R}^2)}
 + |\mathcal{I}'|\cdot \left( \|\alpha_0\|\ +\ \|r\| \right),
 \end{align*}
 where the latter equality holds since $\|u^\varepsilon_{\rm wp}\|_{L^2(\R^2)}=  \|\alpha_0\|_{L^2(\mathbb{R}^2;\mathbb{C}^2)}$, for $\varepsilon$ sufficiently small, by averaging Lemma \ref{lem:homog}.  Therefore, 
 \begin{align*}
   \left(-\frac12 \big|\mathcal{I}^\prime\big| + (\tilde{g}-g)\Tper \right)\|\alpha_0\|_{L^2(\mathbb{R}^2;\mathbb{C}^2)}
& \le C\varepsilon + (3 + |\mathcal{I}^\prime|) \|r\|_{L^2(\mathbb{R}^2)}
\end{align*}
Finally, take $|\mathcal{I}^\prime|< (\tilde{g}-g)\Tper$. By taking $\varepsilon$ sufficiently small we have
\begin{align*}
   C(\tilde{g}-g) 
& \le \|r\|_{L^2(\mathbb{R}^2)} =  \|({\rm I}-{\rm Proj}_{{\rm BL}(d_0,\varepsilon)})u\|_{L^2(\mathbb{R}^2)} . 
\end{align*}
%
 %
 This completes the proof of Theorem \ref{thm:quasigap_general}, with the exception of Proposition \ref{lb*},
which we now establish.

\begin{proof}[Proof of Proposition \ref{lb*}] $u_{wp}^\varepsilon\in~{\rm BL}(d_0,\varepsilon)$ we have (Definition \ref{BLDpkt})
 \[ u_{wp}^\varepsilon(\bx)\ =\ \varepsilon \alpha_0(\varepsilon \bx) \Phi(\bx;\bk _D)^\top\ \textrm{with ${\rm supp}(\hat\alpha_0)\subset\{|\bxi|\le d_0\}$.} \]
In terms of the Fourier transform of $\alpha_0$ and the eigenbasis $\{v_+(\bxi), v_-(\bxi)\}$ for $\C^2$, see \eqref{UDhat-evp}, we have 
\[ \alpha_0(X) =
 \frac{1}{(2\pi)^2}\sum_{r=\pm}\int_{|\bxi|\le d_0}e^{i\bxi X} \left\langle v_r(\bxi),\hat\alpha_0(\bxi)\right\rangle_{\C^2} v_r(\bxi) d\bxi \, ,\]
  and hence
  \begin{equation}
   e^{-i\nu_0}\ u_{wp}^\varepsilon(\bx) = {\rev e^{-i\nu_0}} \Big[ \frac{\varepsilon}{(2\pi)^2}\sum_{r=\pm}\int_{|\bxi|\le d_0}e^{i\bxi (\varepsilon \bx)} \left\langle v_r(\bxi),\hat\alpha_0(\bxi)\right\rangle_{\C^2} v_r(\bxi) d\bxi \Big]^\top \Phi(\bx)\label{nu-wp}
   \end{equation}
By \eqref{UDirac} and \eqref{UDhat} 
\begin{equation} M_{{\rm Dir},d_0}\alpha_0={U}_{\rm Dir}(\Tper,X)[\alpha_0](X)= \frac{1}{(2\pi)^2}\sum_{r=\pm}\int_{|\bxi|\le d_0}e^{i\bxi X} \left\langle v_r(\xi),\hat\alpha_0(\bxi)\right\rangle_{\C^2} e^{i r\mu(\bxi)\Tper}v_r(\bxi) d\bxi \ ,
\label{UD}
\end{equation}
and therefore
\begin{align}
\left(\tMD u_{wp}^\varepsilon\right)(x) 
&= \left[ \varepsilon M_{{\rm Dir},d_0}[\alpha_0](\varepsilon \bx)\right]^\top e^{-iE_D\tper}\ {\Phi}(\bx) \nonumber \\
&= \Big[\ \frac{\varepsilon}{(2\pi)^2}\sum_{r=\pm}\int_{ |\bxi| \le d_0}e^{i\bxi (\varepsilon \bx)} \left\langle v_r(\bxi),\hat\alpha_0(\bxi)\right\rangle_{\C^2} e^{i r\mu(\bxi)\Tper}v_r(\bxi) \ d\bxi \ e^{-iE_D\tper}\ \Big]^\top {\Phi}(\bx). \label{tMD-wp}
\end{align}
Subtracting \eqref{tMD-wp} and \eqref{nu-wp} yields the two-scale expression
   \[ (e^{-i\nu_0}I-\tMD) u_{wp}^\varepsilon(\bx)\ =\ \gamma^\varepsilon(\varepsilon \bx)^\top \Phi(\bx)\, ,\]
   where
   \[ \gamma^\varepsilon(X) = 
  \frac{\varepsilon}{(2\pi)^2}\int_{ |\bxi| \le d_0}e^{i\bxi X} \Big[\sum_{r=\pm}\left\langle v_r(\bxi),\hat\alpha_0(\bxi)\right\rangle_{\C^2}
   \left(e^{-i\nu_0}- e^{-i(E_D\varepsilon^{-1} - r\mu(\bxi))\Tper} \right) v_r(\bxi)\Big] \ d\bxi \ .\]
   We next claim that for all $\varepsilon$ sufficiently small:
   \begin{equation} \left\|(e^{-i\nu_0}-\tMD)u_{wp}^\varepsilon\right\|_{L^2(\mathbb{R}^2)}^2  =  \|\gamma^\varepsilon\|_{L^2(\mathbb{R}^2;\mathbb{C}^2)}^2 \, ,\label{key-lb}\end{equation}
   and therefore the desired lower bound on the norm of $(e^{-i\nu_0}-\tMD)u_{wp}^\varepsilon$ boils down 
   to a lower bound on the norm of $\gamma^\varepsilon(X)$. To prove \eqref{key-lb} we apply averaging Lemma \ref{lem:homog} and use the the orthonormality of  $\Phi_1, \Phi_2 \in L^2_k$.
 Specifically, we have
\begin{align*}
\left\|(e^{-i\nu_0}-\tMD)u_{wp}^\varepsilon\right\|_{L^2(\mathbb{R}^2)}^2 &= 
\left\|\varepsilon{\gamma}^\varepsilon(\varepsilon \bx)^{\top} \Phi(\bx) \right\|_{L^2(\mathbb{R}^2)}^2 = \int\limits_{\mathbb{R}^2} \varepsilon^2 \left| \sum\limits_{j=1}^2 \gamma^\varepsilon_j(\varepsilon \bx) \Phi_j(\bx) \right|^2 \, d\bx\\
&=  \sum\limits_{i, j=1}^2\int\limits_{\mathbb{R}^2} \varepsilon^2  \gamma^\varepsilon_j(\varepsilon \bx)\ \Phi_j(\bx)\ \overline{\gamma^\varepsilon_i(\varepsilon \bx)}\ \overline{\Phi_i(x)}  \,  d\bx \\
&=  \sum\limits_{i, j=1}^2 \langle \gamma^\varepsilon_i, \gamma^\varepsilon _j \rangle_{L^2(\mathbb{R}^2)} \langle \Phi_i, \Phi_j \rangle_{L^2(\Omega)} =  \sum\limits_{i, j=1}^2 \langle \gamma^\varepsilon_i, \gamma^\varepsilon_j \rangle_{L^2(\mathbb{R}^2)} \delta_{i,j}  \\
&=  \sum\limits_{j=1}^2 \|\gamma_i^\varepsilon\|_{L^2(\mathbb{R}^2)}^2 =
 \|\gamma^\varepsilon\|_{L^2(\mathbb{R}^2;\mathbb{C}^2)}^2 \, .
\end{align*}
This completes the proof of \eqref{key-lb}, and thus the desired lower bound reduces to a lower bound on the norm of $\gamma^\varepsilon(X)$. 

By the (vector) Plancherel Theorem and using 
the orthonormality of $\{v_+(\bxi),v_-(\bxi)\}$: 
 \begin{align*}
 \| \gamma^\varepsilon \|_{L^2(\mathbb{R}^2;\mathbb{C}^2)}^2 &= \int_{\R^2} \Big\|\ \sum_{r=\pm}\left\langle v_r(\bxi),\hat\alpha_0(\bxi)\right\rangle_{\C^2}
   \left(e^{-i\nu_0}- e^{-i(E_D\varepsilon^{-1} - r\mu(\bxi))\Tper} \right) v_r(\xi)\ \Big\|_{\C^2}^2\ d\bxi\\
   &= \int_{|\bxi|\le d_0} \Big|\ \sum_{r=\pm}\left\langle v_r(\bxi),\hat\alpha_0(\bxi)\right\rangle_{\C^2}
   \left(e^{-i\nu_0}- e^{-i(E_D\varepsilon^{-1} - r\mu(\bxi))\Tper} \right)\ \Big|^2\ d\bxi\\
   &\ge \min_{r=\pm} \min_{|\bxi|\le d_0}\Big|e^{-i\nu_0}- e^{-i(E_D\varepsilon^{-1} - r\mu(\bxi))\Tper}\Big|^2\ \|\alpha_0\|_{L^2(\mathbb{R}^2;\mathbb{C}^2)}^2
 \end{align*}
By Lemma \ref{lem:homog}, for $\varepsilon$ sufficiently small:  $\|u^\varepsilon_{\rm wp}\|_{L^2(\R^2)}=  \|\alpha_0\|_{L^2(\mathbb{R}^2;\mathbb{C}^2)}$,  and so 
 \[\left\|(e^{-i\nu_0}-\tMD)u_{wp}^\varepsilon\right\|_{L^2(\mathbb{R}^2)}\ge
  \min_{r=\pm} \min_{|\xi|\le d_0}\Big|e^{-i\nu_0}- e^{-i(E_D\varepsilon^{-1} - r\mu(\bxi))\Tper}\Big|\ \|\alpha_0\|_{L^2(\mathbb{R}^2;\mathbb{C}^2)}. \]
Now $\nu_0$ is centered in $\mathcal{I}' \subset \mathcal{I}(g)\subset \mathcal{I}(\tilde{g})$, see Fig.\ \ref{fig:detail}. Hence, \[ \Big|e^{-i\nu_0}- e^{-i(E_D\varepsilon^{-1} \pm\mu(\xi))\Tper}\Big| \ge\frac12 \big|\mathcal{I}^\prime\big| + (\tilde{g}-g)\Tper ,\] 
and therefore, for $\varepsilon>0$ sufficiently small we obtain the lower bound \eqref{eq:gamma_lbd}. This completes the proof of Proposition \ref{lb*}.
\end{proof}

\section{Proof of the Homogenization Lemma \ref{lem:homog}}\label{sec:pf_homog}

Since the fundamental cell of the lattice $\Omega$ tiles the plane, we partition $\mathbb{R}^2$ with respect to the lattice, i.e., $\mathbb{R}^2 = \bigcup_{\bm\in \Lambda} (\Omega +\bm)$.
Therefore 
\begin{align*}
\int\limits_{\mathbb{R}^2}p(\bx) 	q(\varepsilon \bx) \, d\bx &= \sum\limits_{\bm \in \Lambda}\int\limits_{\Omega + \bm}  p(\bx)q(\varepsilon \bx) \, d\bx \\
(\text{change of variables} ~~ \bx=\by+\bm)\qquad  \qquad  &=\sum\limits_{\bm \in \Lambda}\int\limits_{\Omega} p(\by) q(\varepsilon(\by + \bm)) \, d\by \\
&=\int\limits_{\Omega} p(\by)\left[ \sum\limits_{\bm \in \Lambda} q(\varepsilon(\by + \bm))\right] \, d\by \, . \numberthis \label{eq:gpre_pois}
\end{align*}
Using a generalization of Poisson summation formula for general lattices \cite{stein2016introduction}, then
$$\sum\limits_{\bm \in \Lambda} q(\by + \bm) = \sum\limits_{\bn \in \Lambda} \hat{q}(\bn)e^{2\pi i  \by\cdot \bn} \, ,$$
where $\hat{q}$ is the Fourier transform of $q$. Since in two space dimensions we have that $\widehat{q(\varepsilon \cdot)}(\bxi) = \varepsilon^{-2} \hat{q}(\varepsilon ^{-1}\bxi )$, then 
\begin{align*}
\sum\limits_{\bm \in \Lambda} q(\varepsilon(\by + \bm)) &=  \sum\limits_{\bn \in \Lambda} \varepsilon^{-2}\hat{q}\left(\frac{\bn}{\varepsilon}\right)e^{2\pi i  \by\cdot \bn} \, .
\end{align*}
Now, since $q$ is band-limited, then for sufficiently small $\varepsilon$ all of the terms in the last sum vanish but $\bn=(0,0)$. Therefore,
$$\sum\limits_{\bm \in \Lambda} q(\varepsilon(\by + \bm))= \varepsilon^{-2}\hat{q}(0) = \varepsilon^{-2}\int\limits_{\mathbb{R}^2} q(\bx) \,d\bx \, , $$ which when substituted into \eqref{eq:gpre_pois}, yields
\begin{align*}
\int\limits_{\mathbb{R}^2}p(\bx) q(\varepsilon \bx) \, d\bx = \cdots &= \int\limits_{\Omega} p(\by)\left[ \varepsilon^{-2}\int\limits_{\mathbb{R}^2} q(\bx)\, d\bx \right] \, d\by \\
&= \varepsilon^{-2}\left( \int\limits_{\Omega} p(\by)\,d\by \right) \cdot \left( \int\limits_{\mathbb{R}^2} q(\bx) \, d\bx \right) \, . 
\end{align*}

\section{Derivation and time-scale of validity of effective Dirac dynamics : Proof of Theorem \ref{thm:valid} }\label{sec:Dirac}
\subsection{Multiple-scales expansion}\label{sec:ms}
We first derive the Dirac equation \eqref{eq:diracA} using formal multiple scales expansion, and in Section \ref{sec:eta} we  prove Prop.\ \ref{thm:valid}. We introduce slow temporal and spatial scales: $T=\varepsilon t$ and $X = \varepsilon \bx$, and formally view $\psi$ as a function of independent fast and slow variables:
 $\psi=\Psi(x,X,t,T)$. Hence,  $\partial_t\mapsto \partial_t+\varepsilon\partial_T$, 
$\nabla_{\bx}\mapsto \nabla_{\bx}+\varepsilon\nabla_X$ and \eqref{eq:lsA} may be rewritten as:
\begin{align}
 i\left(\partial_t-H^0\right)\Psi\ =\ \varepsilon H_1\Psi + \varepsilon^2 H_2\Psi , \label{ms-eqn}
\end{align}
where
\begin{equation} H_1 = -\left(i\partial_T+2\nabla_{\bx}\cdot\nabla_X+2iA(T)\cdot\nabla_{\bx}\right) 
\quad \textrm{and}\quad  H_2=
 - \left(\Delta_X+2iA(T)\cdot\nabla_X\right).\label{H12}\end{equation}
The procedure we describe will yield $\Psi(x,X,t,T)$ such that 
\[ \textrm{$\Psi(t,T,\bx,X)\Big|_{T=\varepsilon t,\ X=\varepsilon \bx}$ is an approximate solution of \eqref{eq:lsA}.}\] 
We seek a formal expansion for a solution to  \eqref{ms-eqn}, $\Psi(t,T,\bx,X)$,  consisting of a leading multiscale part plus a correction :
\begin{equation}\label{eq:multi_scale}
 \Psi =  \Psi\ord{0}(\bx,X,t,T) +\varepsilon \Psi\ord{1}(\bx,X,t,T) + \varepsilon ^2  \Psi\ord{2}(\bx,X,t,T) + \eta^\varepsilon (t,x) \, .
\end{equation}
Substitution of \eqref{eq:multi_scale} into \eqref{ms-eqn} yields the hierarchy of equations:
\begin{align}
&\mathcal{O}(\varepsilon^0):
(i\partial_t- H^0)\Psi\ord{0}\ =\ 0;\qquad\qquad \Psi(0,\bx,0,X) = \alpha_{0,1}(X)\Phi_1 + \alpha_{0,2}(X)\Phi_2 \label{eq:o0}\\
 \nonumber\\
&\mathcal{O}(\varepsilon^1):
(i\partial_t-H^0)\Psi\ord{1}\ =\ H_1 \Psi\ord{0};\qquad\quad \Psi_1(0,\bx,0,X) = 0 \label{eq:o1} \\
&\nonumber\\
&\mathcal{O}(\varepsilon ^2):(i\partial_t-H^0)\Psi_2\ =\ H_1 \Psi_1 + H_2 \Psi_0;\quad \Psi_2(0,\bx,0,X) = 0  \label{eq:o2} \\
 \nonumber\\ 
 &\textrm{ corrector}: \left(i\partial_t-H^0-2i\varepsilon A(\varepsilon t)\cdot \nabla_{\bx} \right)\eta^{\varepsilon}(t,\bx)\ =\ 
\varepsilon ^3\mathcal{F}^\varepsilon(t,\bx);\qquad \eta(0,\bx) = 0 \, ,\label{eta-cor} 
\end{align}
where
\begin{equation} \mathcal{F}^\varepsilon(t,\bx) =\left(H_1 \Psi _2 +H_2 \Psi_1\right)(t,T,\bx,X) \Big|_{T=\varepsilon t, X=\varepsilon \bx} + \varepsilon \left(H_2 \Psi_2\right)(t,T,\bx,X) \Big|_{T=\varepsilon t, X=\varepsilon \bx}.
 \label{F-def}\end{equation}
Since we are looking to construct solutions of wave-packet type which are spectrally localized at $\bk_{\rm D}=\bK$ we shall seek an expansion where,
for $j=0,1,2$,\[ \textrm{ $\bx\mapsto \Psi_j(t,\bx,X,T)$ is $\bK-$ pseudo-periodic with respect to $\bx$ and decaying to zero as $|X|\to\infty$.}\]
We view equations \eqref{eq:o0}-\eqref{eq:o2} as equations in the space $L^2_\bK$, depending on parameters $T$ and~$X$. The variations with respect to $T$ and $X$ are determined via the solvability conditions of this hierarchy of equations.

We first consider \eqref{eq:o0}. Since the  $L^2_\bK$ nullspace of $ E_D I-H^0$ is given by $\rm{span}\{\Phi_1,\Phi_2\}$,  we have:
\begin{equation}\label{eq:wavepackets}
     \Psi_0\ =\varepsilon \ e^{-iE_Dt}\ \sum_{j=1}^2\ \alpha_j(X,T)\ \Phi_j(x)  \, ,
\end{equation}
where  $\alpha_1(X,T)$ and $\alpha_2(X,T)$ are decaying functions of $X$ to be determined.
The pre-factor of $\varepsilon$ in \eqref{eq:wavepackets} is inserted so that the $L^2(\R^2_x)$ norm is $\mathcal{O}(1)$.

  Turning to first order in $\varepsilon$, equation \eqref{eq:o1}, since $\Psi_0$ is oscillates with frequency $E_D$, it is convenient to define:
\[ \Psi_1(t,T,\bx,X)\ =\ e^{-iE_Dt}\ \tilde{\Psi}_1(T,\bx,X)\, .\]
Substitution into \eqref{eq:o1} yields
\begin{align}
\Big(\ E_D I -H^0\ \Big)\tilde\Psi_1\ &=\ -\ \sum_{j=1}^2\ i\ \partial_T\alpha_j(X,T)\ \Phi_j\ - i\ \sum_{j=1}^2\ \nabla_X\ \alpha_j(X,T)\cdot (-2i)\nabla_{\bx}\Phi_j \nonumber\\
&\quad\quad +\ 
\ \sum_{j=1}^2\ \alpha_j(X,T)\ A(T)\cdot (-2i)\nabla_{\bx}\Phi_j \, . \label{Psi1-eqn}
\end{align}
A necessary condition for the solvability of \eqref{Psi1-eqn} for $\tilde\Psi_1\in L^2_\bK$ is that the right hand side of \eqref{Psi1-eqn} be $L^2_\bK-$ orthogonal to
 $\Phi_1$ and $\Phi_2$.
 These two  solvability conditions read:
 \begin{align*}
i\partial_T\alpha_1\ &=\  -i\ \left\langle\Phi_1,-2i\nabla_{\bx}\Phi_2\right\rangle\cdot \nabla_X\alpha_2\ +\ 
 \left\langle\Phi_1,-2i\nabla_{\bx}\Phi_2\right\rangle\cdot  A(T)\ \alpha_2\ \\
i\partial_T\alpha_2\ &=\  -i \left\langle\Phi_2,-2i\nabla_{\bx}\Phi_1\right\rangle\cdot \nabla_X\alpha_1\ +\ 
 \left\langle\Phi_2,-2i\nabla_{\bx}\Phi_1\right\rangle\cdot  A(T)\ \alpha_1\ .
\end{align*}
By Proposition \ref{ipPhi}
\begin{align}
i\partial_T\alpha_1\ &=\ v_{\rm D}\Big(\ \frac{1}{i}\partial_{X_1}+\partial_{X_2}\ \Big)\ \alpha_2\ +\ 
v_{\rm D}\ \Big(\ A_1(T)\ +\ i A_2(T)\ \Big)\alpha_2\ \, ,\nonumber\\
i\partial_T\alpha_2\ &=\  v_{\rm D}\Big(\ \frac{1}{i}\partial_{X_1}-\partial_{X_2}\ \Big)\ \alpha_1\ +\ 
v_{\rm D}\ \Big(\ A_1(T)\ -\ i A_2(T)\ \Big)\alpha_1 \, , \label{dirac1}
\end{align}
which, after rearranging terms, is seen to be the Dirac equation 
\begin{equation}
i \partial_T \alpha(T,X) = \slashed{D}_{A}(T) \alpha(T,X);\label{dirac1a}
\end{equation}
 see \eqref{Dirac-op}-\eqref{eq:diracA}. 

Let $\pi^{\perp}$ denote the projection onto the orthogonal complement of ${\rm span}\{\Phi _1, \Phi_2 \}$ in $L^2 _{\bK}$:
\[ \pi^\perp = I\ -\  \sum_{j=1}^2 \left\langle\Phi_j,\cdot\right\rangle \Phi_j\ =\ \sum_{j\ge3}\left\langle\Phi_j,\cdot\right\rangle \Phi_j;\] 
for convenience we indexed the $L^2_{\bK}$ eigenpairs of $H^0$ such that
${\rm span}\{\Phi _1, \Phi_2 \}^\perp={\rm span}\{\Phi_j : j\ge3\}$.
  If $(\alpha_1,\alpha_2)$ is constrained to satisfy \eqref{dirac1}, then  $\pi^{\perp} \left( H_1\Psi_0 \right) = H_1\Psi_0$ and hence:
\begin{equation}\label{eq:psi1_exp}
\tilde{\Psi}_1(t,\bx,T,X) = (E_D I - H^0)^{-1} H_1\Psi_0 +\varepsilon \sum\limits_{j=1,2}\beta _j(T,X) \Phi_j (x) \, ,
\end{equation}
where $\varepsilon\beta_j$ are decaying functions of $X$ which are to be determined at the next order equation \eqref{eq:o2}, and finally
$\Psi_1 =~\tilde{\Psi}_1 e^{-iE_D t}$. Turning next to second order in $\varepsilon$, \eqref{eq:o2},  we write $\Psi_2(t,\bx,T,X) = \tilde{\Psi}_2(\bx,T,X) e^{-iE_Dt}$. In analogy with our first order analysis, the condition for solvability condition of \eqref{eq:o2} in $L^2_\bK$ is that $H_1 (\tilde{\Psi }_1+\sum_j \beta _j \Phi_j )$ is orthogonal to $\Phi_1$ and~$\Phi_2$. In a manner analogous  to the derivation of \eqref{eq:diracA}, we obtain a system of {\em forced} Dirac equations for $\beta \equiv (\beta_1,\beta_2)^{\top}$:
\begin{equation}\label{eq:dirac2}
i\partial_T \beta(T,X) -  \slashed{D}_A(T)\beta(T,X) = F_2(T,X) \, , 
\end{equation}
where $F_2=\left(F_{2,1} , F_{2,2} \right)^\top$, and  for $j=1,2$:
\begin{equation} \qquad F_{2,j} = \langle \Phi_j, H_1 (E_D-H^0)^{-1} H_1 \Psi_0 \rangle_{L^2(\Omega)} \, .
\end{equation}
We note that $F_2$ is independent of $\beta$, and is therefore a forcing term in \eqref{eq:dirac2}. 
Corresponding to any solution of the initial value problem for \eqref{eq:dirac2} in $C(\R;H^s(\R))$,  we have that 
\begin{equation}\label{eq:psi_2}
\Psi_2 (t,\bx,T,X) = e^{-iE_Dt} (E_D-H^0)^{-1}\pi^{\perp}\tilde{\Psi}_1 \, .
\end{equation}

\subsection{Bounding the corrector, $\eta^\varepsilon(t,\bx)$}\label{sec:eta}

Finally we turn to estimating the corrector, $\eta^\varepsilon$, which \eqref{eta-cor}, \eqref{F-def}.
By self-adjointness of $H^0 +~iA(\varepsilon t)\cdot~\nabla$, we have 
$ \partial_t \|\eta^{\varepsilon}(t)\|^2 = 2\varepsilon^3 {\rm Re}\ \langle \eta , \mathscr{F}^\varepsilon(t,\cdot) \rangle$.
This implies, by the Cauchy-Schwarz inequality, that
$\partial_t \|\eta^{\varepsilon}(t)\| \le  \varepsilon^3  \|\mathscr{F}^\varepsilon(t,\cdot)\|$. Therefore, for all $t\ge0$:
\begin{equation}\label{eq:eta_bdF}
\|\eta ^{\varepsilon}(t,\cdot) \|_{L^2(\R ^2)} \leq t\varepsilon ^3 \sup\limits_{s\in [0,t]}\|\mathscr{F}^\varepsilon(s, \cdot)\|_{L^2(\R ^2)} \, ,
\end{equation}
where $\mathscr{F}^\varepsilon$ is given by \eqref{F-def}. To prove Theorem \ref{thm:valid} we next bound  $\|\mathscr{F}^\varepsilon(t, \cdot)\|$ for $t\lesssim \varepsilon^{-2+}$.
\bigskip

\begin{proposition}\label{lem:bding_ab}
For all $t>0$, we have the following $L^2 (\mathbb{R}^2)$ bounds
\begin{align}
\|\eta^{\varepsilon}(t, \cdot)\|_{_{L^2 (\mathbb{R}_x^2)}} &\lesssim t\varepsilon^3 \sup\limits_{s\in [0,t]}\left(\|\alpha(\varepsilon s, \cdot) \|_{_{H^3(\mathbb{R}^2)}}  + \|\beta(\varepsilon s, \cdot)\|_{_{H^2(\mathbb{R}^2)}}\right) \, . \label{eq:eta_ab_bd}
\end{align}
Here, $\alpha = (\alpha _1 , \alpha _2 )^{\top}$ and $\beta = (\beta _1, \beta _2 ) ^{\top}$ are solutions of
 the homogeneous and inhomogeneous Dirac equations  \eqref{dirac1a} and \eqref{eq:dirac2}, respectively.
The implicit constant in \eqref{eq:eta_ab_bd} depends only on the honeycomb potential $V$. 
\end{proposition}

%
%
\begin{proof} We shall use the following convention. If $G(t,T,\bx,X)$ is a multi-scale function, then we write $G^\varepsilon(t,\bx)\equiv G(t,T,\bx,X)\Big|_{T=\varepsilon t, X=\varepsilon \bx}$. The proof of Proposition \ref{lem:bding_ab} makes use of the following bounds:

\begin{align}
\|\Psi_1^{\varepsilon}(t,\cdot )\|_2 &\lesssim \|\alpha(\varepsilon t,\cdot) \|_{H^1} + \|\beta(\varepsilon t,\cdot) \|_2  \, , \label{eq:psi1_bd_ab} \\
\|\Psi_2^{\varepsilon}(t,\cdot ) \|_2 &\lesssim \|\alpha(\varepsilon t, \cdot) \|_{H^1} + \|\vec{\beta}(\varepsilon t, \cdot) \|_2  \, , \label{eq:psi2_bd_ab} 
\\
\|H_1 \Psi_2 ^{\varepsilon}(t,\cdot )\|_2 &\lesssim \|\alpha(\varepsilon t, \cdot)\|_{H^2} + \|\vec{\beta}(\varepsilon t, \cdot)\|_{H^1} \, ,\label{eq:H1psi2_bd_ab}
\\
\|H_2 \Psi_1^{\varepsilon}(t,\cdot ) \|_2 &\lesssim\|\alpha(\varepsilon t, \cdot)\|_{H^3} + \|\vec{\beta}(\varepsilon t, \cdot)\|_{H^2} \, ,\label{eq:H2psi1_bd_ab}
\\
\|H_2 \Psi_2 ^{\varepsilon}(t,\cdot )\|_2 &\lesssim \|\alpha(\varepsilon t, \cdot)\|_{H^3} + \|\vec{\beta}(\varepsilon t, \cdot)\|_{H^2} \, ,\label{eq:H2psi2_bd_ab}\\
\|\mathcal{F}^{\varepsilon}(\varepsilon t,\cdot)\|_2 &\lesssim \|\alpha (\varepsilon t, \cdot)\|_{H^3} + \|\vec{\beta}(\varepsilon t, \cdot)\|_{H^2}  \, . \label{eq:F_ab_bd}
\end{align}
It will be useful to decompose $\tilde{\Psi}_1$ into two separate terms and bound each of these elements separately
$$\tilde{\Psi}_{11}=(E_D-H)^{-1}\pi^{\perp}H_1\Psi_0 \, , \qquad \tilde{\Psi}_{12} = \sum_j \beta _j \Phi_j \,  .$$
 We start with bounding $\|\tilde{\Psi}_{1,1}\|_2$. 
By definition
\begin{align*}
-H_1 \Psi _0 &= \varepsilon \sumj  i\partial_T \alpha_j(T,X)\Phi_j +2\nabla _x \Phi_j \cdot \nabla _X \alpha_j(T,X)  -2i\alpha _j(T,X)  A  \cdot \nabla \Phi_j 
\end{align*}
where $ {\Phi} = (\Phi_1, \Phi_2 )^{\top}$. Since $(E_D I -H)^{-1}$ operates on $L^2_{\bK}$, represent  $(E_D I -H)^{-1} \pi^\perp H_1 \Psi_0$ in the basis of the Bloch modes $\{\Phi_b\}_{b\geq3}$. Using that, by the derivation of \eqref{dirac1a},
 the constraint $i\partial_T\alpha(T,X)=\slashed{D}(T)\alpha(T,X)$ enforces  $\pi^\perp H_1\Psi_0=H_1\Psi_0$, we have
\begin{align*}
&\tilde{\Psi}^\varepsilon_{1,1}(t,\bx)=\tilde{\Psi}_{1,1}(t,\varepsilon t,\bx,\varepsilon \bx)\\
&=(E_D-H)^{-1}H_1 \Psi_0\Big|_{T=\varepsilon t, X=\varepsilon \bx}\\ &=\varepsilon\sum\limits_{b\geq 3} \sumj \frac{\Phi_b }{E_b -E_D}\langle \Phi_b ,  i\partial_T \alpha_j(T,X)\Phi_j +2\nabla _x \Phi_j \cdot \nabla _X \alpha_j(T,X) -i\alpha_j(T,X) A(T) \cdot \nabla_{\bx} \Phi_j \rangle_{L^2(\Omega)} \\
&=\varepsilon \sum\limits_{b\geq 3} \frac{\Phi_b }{E_b -E_D}\\
&\qquad\quad  \left[\langle \Phi_b , \Phi \rangle_{L^2(\Omega)}\cdot  \slashed{D}(T) \alpha(T,X) +\sumj\langle \Phi_b ,\nabla_{\bx} \Phi_j \rangle_{L^2(\Omega)}\cdot (2\nabla _X \alpha(T,X) -2 iA(T) \alpha_j(T,X) ) \right]_{T=\varepsilon t, X=\varepsilon \bx} \, . \\
\end{align*}
We estimate $\tilde{\Psi}^\varepsilon_{1,1}$ in $L^2=L^2(\mathbb{R}^2)$ using that
\begin{enumerate}
\item Since $\Phi_b \in L^{\infty}$, then $\|\Phi_b(\cdot) f(\cdot) \|_2 \lesssim \|f\|_2$ for any $f\in L^2$.
\item $ \|f(\varepsilon \cdot ) \|_{H^s}  \lesssim \varepsilon ^{-1} \|f\|_{H^s}$ for every $s\geq 0$.  
\end{enumerate} 
Hence,
\begin{align}
\|\tilde{\Psi}^\varepsilon_{1,1}(t,\cdot) \|_{L^2(\mathbb{R}^2)} & \lesssim  \varepsilon  \sum\limits_{b\geq 3}\sumj \frac{\|\Phi_b \|_{L^{\infty}} }{|E_b -E_D|} \big|\langle \Phi_b , \Phi_j \rangle_{_{L^2(\Omega)}}\big|\ \varepsilon^{-1}\|  \nabla_X\alpha_j (\varepsilon t,\cdot)\|_{L_X^2}\nonumber\\
&\qquad\quad + \varepsilon  \sum\limits_{b\geq 3} \frac{\|\Phi_b\|_{L^{\infty}} }{|E_b -E_D|}| \big|\langle \Phi_b , \nabla  \Phi_j\rangle _{_{L^2(\Omega)}}\big|\ \varepsilon^{-1}\|\alpha_j(\varepsilon t,\cdot)\|_{H_X^1}\nonumber  \\ 
&\lesssim  \|\alpha(\varepsilon t, \cdot) \|_{H_X^1} \sumj \sum\limits_{b\geq 3}  \frac{\|\Phi_b\|_{\infty} }{|E_b -E_D|}\left( 
\big|\langle \Phi_b , \Phi_j \rangle_{_{L^2(\Omega)}}\big| + |\langle \Phi_b , \nabla  \Phi_j\rangle _{_{L^2(\Omega)}}| \right) \, .
\label{Psi11-est}\end{align}
 By the Sobolev inequality \cite{evans1998partial} for $\Omega\subset\R^2$, and the relation $\Delta \Phi_b = (V-E_b)\Phi_b$, we have for any $b\ge1$:
$$\|\Phi_b \|_{L^{\infty}(\Omega)} \lesssim \|\Phi _b \| _{H^2(\Omega ) } \lesssim \|\Phi_b \|_{L^2(\Omega)} + \|\Delta \Phi_b \| _{L^2(\Omega)} =  \|\Phi_b \|_{L^2(\Omega)} + \|(V-E_b)\Phi_b \| _{L^2(\Omega)}\, .$$
Furthermore, since $V$ is bounded  
$\|\Phi_b \|_{L^{\infty}(\Omega)} \lesssim E_b \|\Phi_b\|_{L^2(\Omega)} \lesssim |b| \, $, 
where we have used  that $\|\Phi _b\|=1$, and that $E_b \sim b$ as $b\to \infty$, by Weyl asymptotics  in two dimensions. Hence, the factor $\|\Phi _b \|_{\infty} / |E_b- E_D|$ in \eqref{Psi11-est} is uniformly bounded for all $b$.
Therefore,  bounding $\|\tilde{\Psi}^\varepsilon_{1,1} \|_{L^2(\mathbb{R}^2)}$
 reduces to showing, for $j,l=1,2$:  
 \[\sum_{b\ge3} \big|\langle \Phi _b , \Phi _j \rangle_{_{L^2(\Omega)}}\big| +\big| \langle \Phi_b, \partial_{x_l} \Phi_j  \rangle _{_{L^2(\Omega)}} \big| < \infty\]
 We claim that both summands decay rapidly with $b$. Indeed, by the self-adjointedness of $H^0$, for $r=0,1$:
 \[
 \langle \Phi_b, \partial^r_{x_l} \Phi_j  \rangle = E_b^{-2}\langle (H^0)^2\Phi_b, \partial^r_{x_l} \Phi_j  \rangle
 =
E_b^{-2}\langle \Phi_b, (H^0)^2\partial^r_{x_l} \Phi_j  \rangle
 \]
 and therefore
 \[  \big|\langle \Phi_b, \partial^r_{x_l} \Phi_j  \rangle\big| \lesssim b^{-2} \| (H^0)^2\partial^r_{x_l} \Phi_j \| 
 \lesssim b^{-2} \|\Phi_j\|_{_{H^{4+r}}},\ r=0,1, \]
 which is sufficient to ensure summability. It follows that 
\begin{equation}\label{eq:psi11_bd}
\|\tilde{\Psi}^\varepsilon_{1,1}(t,\cdot )\|_{L^2(\mathbb{R}_x^2)} \lesssim \|\alpha(\varepsilon t,\cdot) \|_{H^1(\R^2_X)} \, .
\end{equation}
Together with the bound $\|\tilde{\Psi}_{1,2}(t,\cdot ) \|_2 \lesssim \varepsilon \|\beta (\varepsilon t,\varepsilon \cdot) \|_{L^2(\mathbb{R}_x^2)} \lesssim \|\beta(\varepsilon t,\cdot) \|_{L^2(\mathbb{R}_X)}$, we obtain \eqref{eq:psi1_bd_ab}. 
 
 The upper bounds \eqref{eq:psi2_bd_ab}--\eqref{eq:H2psi2_bd_ab} proceed in a similar fashion. The upper bound \eqref{eq:F_ab_bd} follows directly from the triangle inequality and $\mathscr{F} = H_1\Psi_2 + H_2\Psi_1 + \varepsilon H_2 \Psi_2$. Finally, we prove \eqref{eq:eta_ab_bd} by combining \eqref{F-def},  \eqref{eq:eta_bdF}, and \eqref{eq:F_ab_bd}.

\end{proof}

Proposition \ref{lem:bding_ab} provides upper bounds for $\eta^\varepsilon$, the expansion corrector, in terms of the Sobolev norms of $\alpha(T,X)$ and $\beta(T,X)$, which satisfy the Dirac equations \eqref{eq:diracA} and \eqref{eq:dirac2}.  We now turn to estimating these norms. 
\begin{lemma}\label{lem:dirac_bds}
Let $\alpha$ satisfy and $\beta$ denote solutions of homogeneous and non-homogeneous Dirac equations
 \eqref{eq:diracA} and \eqref{eq:dirac2}. As initial data we take
 $\alpha(0,\cdot)=\alpha_0 \in H^4(\mathbb{R}^2;\mathbb{C}^2)$ and $ \beta(0,x) \equiv 0$.
Then, for all $T>0$ and $s\in \mathbb{N}$, 
\begin{align}
\|\alpha(T,\cdot ) \|_{H^s} &= \|\alpha(0,\cdot) \|_{H^s} \, , \label{eq:alpha_hs_bd} \\
\|\beta(T,\cdot )\|_{H^s} &\lesssim T \|\alpha(0, \cdot)\|_{H^{s+2}} \, . \label{eq:beta_hs_bd}
\end{align}
\end{lemma}
\begin{proof}  The conservation law  \eqref{eq:alpha_hs_bd} follows from unitarity and translation invariance; see \eqref{unitarity}. To obtain \eqref{eq:beta_hs_bd}, we have from \eqref{eq:dirac2}, in $L^2(\mathbb{R}^2)$,  that
\begin{equation}
i\partial_T \|\partial_X^s \beta ( T, \cdot) \|_2 ^2 = 2~{\rm Im}\langle \beta ( T, \cdot) , \partial_X^{2s}F_2 \rangle =(-1)^s \cdot 2~{\rm Im}(\langle \partial_X^{s} \beta( T, \cdot)  , \partial_X^s F_2 \rangle) \, .
\end{equation} 
Therefore, by the Cauchy-Schwarz inequality, 
$\partial_T \|\partial_X^s \beta ( T, \cdot) \|_2 \leq \|F_2 (T, \cdot)\|_{H^s (\R ^2)} \, .
$
Finally, we bound the $H^s(\mathbb{R}^2)$ norm of $F_2$ .
Since, by definition
$$F_{2,j} = \langle \Phi_j, H_1 (E_D-H^0)^{-1}H_1\Psi_0 \rangle_{L^2_x(\Omega)} \, ,$$
we, as in Lemma \ref{lem:bding_ab}, using \eqref{eq:alpha_hs_bd} that  $\|F_2 (T, \cdot)\|_{H^s_x} \lesssim \|\alpha  (T, \cdot)\|_{H^{s+2}_X} = \|\alpha_0 \|_{H^{s+2}_X}$. 
Thus, $$\|\partial_X^s \beta \|_{L^2(\R ^2)} \lesssim \int\limits_{0}^{\top} \|F_2 (T', \cdot ) \|_{H^{s+2}_X} \, dT' \lesssim T\|\alpha_0 \|_{H^{s+2}_X},$$
which proves \eqref{eq:beta_hs_bd}.
\end{proof}

We now complete the proof of Theorem \ref{thm:valid}. With the notation and definitions introduced above and \eqref{eq:multi_scale}
our solution of \eqref{eq:lsA} is:
\[ \psi^\varepsilon(t, \bx) = \Psi^\varepsilon_0 (t,\bx)  + \varepsilon\Psi^\varepsilon_1 (t,\bx) + \varepsilon^2 \Psi^\varepsilon_2(t,\bx) + \eta^\varepsilon(t,\bx). \]
We shall estimate the size of the corrector to the leading  order (effective Dirac) approximation:
$$\psi(t,\bx) - \Psi^\varepsilon_0 (t,\bx)= \varepsilon\Psi^\varepsilon_1 (t,\bx) + \varepsilon^2 \Psi^\varepsilon_2(t,\bx) + \eta^\varepsilon(t,\bx) .$$
Using Lemmas \ref{lem:bding_ab} and \ref{lem:dirac_bds} we have that
\begin{align*}
\|\psi^\varepsilon(t,\bx)- \Psi_0 (t,\cdot)\|_2  &\lesssim \varepsilon\|\Psi_1 (t,\cdot)\|_2 + \varepsilon^2 \|\Psi_2 (t,\cdot )\|_2 + \|\eta(t,\cdot)\|_2 \\
&\lesssim \varepsilon  \underbrace{(\|\alpha (\varepsilon t, \cdot)\|_{H^1} + \|\beta(\varepsilon t,\cdot)\|_2)}_{ \Psi_1} + \varepsilon^2\underbrace{(\|\alpha(\varepsilon t, \cdot)\|_{H^1} +\|\beta(\varepsilon t,\cdot)\|_2)}_{ \Psi_2} +\underbrace{t\varepsilon^3 (\|\alpha(\varepsilon t, \cdot) \|_{H^3} + \|\beta(\varepsilon t,\cdot)\|_{H^2})}_{\eta} \\
&\lesssim \varepsilon  \underbrace{(\|\alpha_0\|_{H^1} + \|\beta(\varepsilon t,\cdot)\|_2)}_{ \Psi_1} + \varepsilon^2\underbrace{(\|\alpha_0\|_{H^1} +\|\beta(\varepsilon t,\cdot)\|_2)}_{ \Psi_2} +\underbrace{t\varepsilon^3 (\|\alpha_0 \|_{H^3} + \|\beta(\varepsilon t,\cdot)\|_{H^2})}_{\eta} \\
&\lesssim \varepsilon \left( \|\alpha_0\|_{H^1} + \varepsilon t \|\alpha_0 \|_{H^2} \right) +\varepsilon^2 \left( \|\alpha_0 \|_{H^1} + \varepsilon t \|\alpha_0 \|_{H^2} \right) + t\varepsilon^3 \left(\|\alpha_0 \|_{H^3} +t\varepsilon \|\alpha _0\|_{H^4} \right) \\
&\lesssim \varepsilon \|\alpha_0 \|_{H^1} + \varepsilon^2 t \|\alpha_0 \|_{H^2} +t\varepsilon ^3 \|\alpha_0 \|_{H^3} + t^2\varepsilon ^4 \|\alpha_0 \|_{H^4} \, .
\end{align*}
Therefore, for any $\rho>0$ and $\varepsilon$ sufficiently small,  $\sup_{0\leq t\lesssim \varepsilon^{-(2-\rho)}} \|\psi^\varepsilon(t,\bx)- \Psi_0 (t,\cdot)\|_2 \lesssim \varepsilon^{\rho}$. This completes the proof of Theorem \ref{thm:valid}.

\section{Proof of Proposition \ref{prop:wp_dk}}\label{sec:proj}
\subsection*{From projections to wave-packets; proof of \eqref{wp_dk1}}


Let $\varepsilon >0$ be taken sufficiently small, and let $f\in L^2(\mathbb{R}^2)$. Express $H$ acting in $ L^2(\mathbb{R}^2)$ as a direct integral $H = \int^\oplus_{\mathcal{B}} H_k \, d\bk$,  where  $H_\bk $ denotes the operator $H = -\Delta +V$ acting in $ L^2_\bk$.   Then under the no-fold condition (Definition \ref{def:nofold}) there exists a constant $a>0$, which depends on $v_{\rm D}$, such that
\begin{align*}
{\rm Proj}_{L^2(\mathbb{R}^2)}(|H- E_D|<\varepsilon) f &=  {\rev \sum\limits_{\bk_{\rm D} =\bK, \bK'} }\int\limits_{\mathcal{B}} \, d\bk \  \chi \left(\frac{|\bk-\bk _D|}{\varepsilon}<a\right) {\rm Proj}_{L^2_\bk}(|H _\bk -E_D|<2\varepsilon) \, f  \\
&= {\rev  \sum\limits_{\bk_{\rm D} =\bK, \bK'}}\int\limits_{\mathcal{B}} \, d\bk \,  \chi \left(\frac{|\bk-\bk _D|}{\varepsilon}<a\right) \left[\frac{1}{2\pi i}\oint\limits_{|\zeta - E_D|=2\varepsilon} \,   (\zeta I   - H_\bk)^{-1}\ d\zeta \right] \, f  \, . \numberthis \label{eq:proj_cont} \\
\end{align*}
\noindent
{\rev We next expand the  terms in \eqref{eq:proj_cont} for $\bk$ near  $\bk_D$, focusing on the $\bk_{\rm D}=\bK$ term; he $\bk_{\rm D}=\bK'$ term is treated analogously.} In order to expand for $\bk$ near $\bK$,  we next express the operators $H_\bk$ in terms of operators which acts in the fixed space $L^2_{\bk _D}$.
Note that  $H_\bk = e^{i\bk \cdot  \bx} H(\bk) e^{-i\bk \cdot  \bx}$, where $H(\bk)\equiv -(\nabla +i\bk)^2 + V$  acts in $L^2(\R^2/\Lambda)$.
Furthermore, $ (\zeta I  - H_\bk)^{-1}=e^{i\bk \cdot  \bx}(\zeta I- H(\bk) )^{-1} e^{-i\bk \cdot  \bx}  $.

Substitution into \eqref{eq:proj_cont} yields
\begin{align*}
\cdots
&=  \int\limits_{\mathcal{B}} \, d\bk \,  e^{i\bk \cdot  \bx}\chi \left(\frac{|\bk-\bK|}{\varepsilon}<a\right) \left[\frac{1}{2\pi i}\oint\limits_{|\zeta - E_D|=2\varepsilon}\  (\zeta I - H(\bk))^{-1}\ \, d\zeta \right] \, e^{-i\bk \cdot  \bx}f  \\
\qquad &=  \int\limits_{\mathcal{B}} \, d\kappa  \,  e^{i(\bK+\kappa )\cdot \bx}\chi \left(\frac{\kappa}{\varepsilon} <a\right) \left[\frac{1}{2\pi i}\oint\limits_{|\zeta - E_D|=2\varepsilon} \,  (\zeta I   - H(\bK+\kappa ))^{-1}\ d\zeta \,  \right] \, e^{-i(\bK+\kappa)\cdot \bx}f   \, .
\end{align*}
The contour integral inside the square brackets is smooth $L^2(\R^2/\Lambda)$-valued function of $\kappa$, and so by Taylor expansion:
\begin{align}
\cdots  = & \int\limits_{\mathcal{B}} \, d\kappa  \,  e^{i(\bK+\kappa )\cdot \bx}\chi \left(\frac{|\kappa|}{\varepsilon} <a\right) \left[\frac{1}{2\pi i}\oint\limits_{|\zeta - E_D|=2\varepsilon} \,  (\zeta I  - H(\bK))^{-1}  \, d\zeta \right] \, e^{-i(\bK+\kappa)\cdot \bx}f  \nonumber\\
&+  \int\limits_{\mathcal{B}} \, \chi \left(\frac{\kappa}{\varepsilon} <a\right) \ \kappa\ \mathcal{\rm Error}[f;\kappa]\ d\kappa  \, .
\label{texp}\end{align}
The last term in \eqref{texp} is linear in $f$ and easily seen to be bounded in $L^2(\R^2)$ by $\varepsilon^3 \|f\|_{L^2}$ since the domain of integration
 is over a disc of radius $\varepsilon$.  
 
The dominant term in \eqref{texp} may be re-expressed as
 \begin{align*}
 & \int\limits_{\mathcal{B}} \, d\kappa  \,  \chi \left(\frac{|\kappa|}{\varepsilon} <a\right) 
e^{i\kappa\cdot \bx}\left[\frac{1}{2\pi i}\oint\limits_{|\zeta - E_D|=2\varepsilon} \,  e^{i\bK\cdot \bx}(\zeta I  - H(\bK))^{-1}e^{-i\bK\cdot \bx}  \, d\zeta \right] \, e^{-i\kappa\cdot \bx}f (\bx)\nonumber\\
\quad  = & \int\limits_{\mathcal{B}} \, d\kappa  \,  \chi \left(\frac{|\kappa|}{\varepsilon} <a\right) 
e^{i\kappa\cdot \bx}\left[\frac{1}{2\pi i}\oint\limits_{|\zeta - E_D|=2\varepsilon} \,  
(\zeta I  - H_{\bK})^{-1}  \, d\zeta \right] \, e^{-i\kappa\cdot \bx}f(\bx) \nonumber\\
= &\int\limits_{\mathcal{B}} \, d\kappa  \,  \chi \left(\frac{|\kappa|}{\varepsilon} <a\right) 
e^{i\kappa\cdot \bx}\ 
 {\rm Proj}_{ L^2_{\bK} }(|H _{\bK} -E_D|<2\varepsilon) \ e^{-i\kappa\cdot \bx}f(\bx)\\
&= \int\limits_{\mathcal{B}}\ d\kappa\  \chi\left(\frac{|\bkappa|}{\varepsilon} < a\right)\ e^{i\kappa \cdot \bx}\ \Phi^{\top}(\bx; \bK) \left[\int\limits_{\mathbb{R}^2} \, d\by \, \overline{\Phi (\by;\bK)}f(\by)e^{-i\kappa \cdot \by} \right]   \nonumber\\
&= \Phi^{\top}(\bx;\bK) \int\limits_{\mathbb{R}^2} \, d\by \, \overline{\Phi(\by;\bK)} (\by)f(\by)  \left[ \int\limits_{\mathcal{B}} \, d\kappa \,  \chi\left(\frac{|\bkappa|}{\varepsilon} < a\right)e^{i \bkappa\cdot  (\bx-\by)} \right]  \nonumber \\
&= \Phi^{\top} (\bx;\bK)\int\limits_{\mathbb{R}^2} \, d\by \, \overline{\Phi (\by;\bK)}f(\by)  \left[ \int\limits_{\mathbb{R}^2} \, d\kappa \,  \chi\left(\frac{|\bkappa|}{\varepsilon} < a\right)e^{i\bkappa \cdot (\bx-\by)} \right]  \nonumber 
\end{align*}

Adding the analogous $\bK'$ term from \eqref{eq:proj_cont}, we have that 
 \begin{align}
 {\rm Proj}_{L^2(\mathbb{R}^2)}(|H - E_D|<\varepsilon) f  &= {\rev \sum\limits_{\bk_{\rm D}  = \bK, \bK'}} {\rev u_\varepsilon^{\bk_{\rm D}}[f]}+ \mathcal{O}_{L^2(\R^2)}\left(\varepsilon^3 \|f\|_{L^2}\right),\quad {\rm where}
 \label{texp1} \end{align}
  \begin{align}
 u_\varepsilon^{\bk_{\rm D}}[f](\bx) &\equiv  \Phi^{\top}(\bx;\bk_{\rm D}) \beta_\varepsilon^{\bk_{\rm D}} [f](\bx),\quad {\rm and}  \nonumber\\
 \beta^\varepsilon_{\bk_{\rm D}}[f](\bx) &\equiv \left[\ \left( \overline{\Phi(\cdot;\bk_{\rm D})} f \right) \ast \mathcal{F}_{\bxi}^{-1}\left[ \chi \left( \frac{|\bxi|}{\varepsilon}<a \right)\right]\ \right](\bx) \, ,
 \label{u-beta}
\end{align}
where $\mathcal{F}^{-1}[g](\bxi)$ denotes the inverse Fourier transform, {\rev $\bk_{\rm D} =\bK, \bK'$}, and $\ast$ denotes convolution. We next show that $u_\varepsilon^{\bk_{\rm D}}[f]\in {\rm BL}(a,\varepsilon; \bk_{\rm D})$ by showing that 
$\mathcal{F}[\ \beta_\varepsilon^{\bk_{\rm D}}[f]\ ] \in\chi (|\bxi|< a\varepsilon)L^2_{\bxi } $.
 Indeed,  \[
\mathcal{F}[ \beta^\varepsilon_{f,\bk_{\rm D}}](\bxi) = \mathcal{F}\left[\left( \overline{\Phi(\cdot ;\bk_{\rm D})} f \right) \ast \mathcal{F}^{-1}\left[ \chi \left( \frac{|\bxi|}{\varepsilon}<a \right)\right]\right] = \mathcal{F}[\overline{\Phi(\cdot ;\bk_{\rm D})} f](\bxi) \cdot\chi\left(\frac{|\bxi|}{\varepsilon} <a \right) ,
\]
which is supported in $\{|\bxi|<\varepsilon a\}$. This completes the proof of \eqref{wp_dk1}.
\subsection*{From wave-packets to projections; proof of \eqref{wp_dk2}}
{\rev Assume,  without loss of generality, that 
$u(\bx)\in {\rm BL}(d_0,\varepsilon;\bK)$ for some $d_0, \varepsilon > 0$, i.e., that  
\[ u(\bx)=~\Phi^{\top}(\bx;\bK)\alpha_\varepsilon (\bx),\quad \textrm{where}\quad 
\widehat{\alpha_\varepsilon}(\bxi) = \chi\left(\frac{|\bxi|}{\varepsilon}<a\right)\widehat{\alpha_\varepsilon}(\xi).\]
 Then by \eqref{texp1}, there exists $a>0$ and such that for $\varepsilon>0$ sufficiently small we have that 
\begin{equation}\label{eq:Proj_BL}
{\rm Proj}_{L^2(\mathbb{R}^2)} (|H^0-E_D|\leq 
\varepsilon)u =  \sum\limits_{\bk_{\rm D}=\bK, \bK'}\Phi^{\top}(\bx, \bk_{\rm D})\gamma_{\varepsilon,\bk_{\rm D}} (\bx) + \mathcal{O}(\varepsilon^3\|u\|_{L^2(\mathbb{R}^2)}) \, .
\end{equation}
To prove \eqref{wp_dk2} it suffices to show that 
$\gamma_{\varepsilon,\bK}(\bx)= \alpha_\varepsilon(\bx)$ and $\gamma_{\varepsilon,\bK^\prime}(\bx)=0$. 
Substitution of 
 $u=\Phi^{\top}(\bx;\bK)\alpha_{\varepsilon}(\bx)$ into \eqref{u-beta} yields
$$ \gamma_{\varepsilon,\bk_{\rm D}} = 
\left(\overline{\Phi(\cdot;\bk_{\rm D})}\Phi^{\top}(\cdot;\bK)\alpha_\varepsilon \right) \ast\mathcal{F}^{-1}\left[\chi\left(\frac{|\bxi|}{\varepsilon}<a \right)\right] \, , \qquad \bk_{\rm D} = \bK, \bK' \, . $$
We next  compute the Fourier transform of $\gamma_{\varepsilon,\bK_{\rm D}}$. For $j=1,2$, and $\bk_{\rm D} = \bK, \bK'$
\begin{align}
    \mathcal{F}[ \gamma_{\varepsilon,\bk_{\rm D}}]_j &=\mathcal{F}[\overline{\Phi(\bx; \bk_{\rm D})}\Phi^{\top}(\bx; \bK)\alpha ( \bx) ]_j ~ \cdot \chi\left(\frac{|\bxi|}{\varepsilon}<a \right) \nonumber\\
    &= \sum\limits_{\ell =1,2} \mathcal{F}[\overline{\Phi_j (\bx;\bk_{\rm D})}\Phi_{\ell}(\bx; \bK) \alpha_{\varepsilon, \ell}(\bx)] ~ \cdot\chi\left(\frac{|\bxi|}{\varepsilon}<a \right) \label{eq:gammaFour_terms} 
\end{align}
Consider the expression being summed in \eqref{eq:gammaFour_terms}. Since $\Phi_j(\bx;\bk_D)=e^{i\bk_D\cdot\bx}\phi_j(\bx;\bk_D)$ with
 $\phi_j(\bx;\bk_D)\in L^2(\R^2/\Lambda)$ periodic, we have
  \begin{align*}
 \overline{\Phi_j(\bx, \bk_D)}\Phi_{\ell}(\bx, \bK) \alpha_{\varepsilon, \ell}^\varepsilon(\bx) 
= p_{j,\ell}(x)q_{ \varepsilon, \ell}(x)\, ,
 \end{align*}
 where $p_{j,\ell}(\bx)=\overline{\phi_j(\bx, \bk_D)}\phi_{\ell}(\bx, \bK) \in L^2(\R^2/\Lambda)$ and
$ q_{\varepsilon,\ell}(\bx,\xi) = \alpha_{\varepsilon, \ell}(\bx) e^{i\bx\cdot (\bK-\bk_D)}$. Expanding $p_{j,\ell}(\bx)$ in a  Fourier series $p_{j,\ell}(x) = \sum_{\bn \in \Lambda^{*}} \hat{p}_{j,\ell}(\bn) e^{i\bn \cdot \bx}$ and substituting in \eqref{eq:gammaFour_terms} yields
\begin{align*}
\mathcal{F}[\gamma_{\varepsilon,\bk _{\rm D}}]_j &= \sum\limits_{\ell = 1,2}\mathcal{F} \left[ \sum\limits_{\bn \in \Lambda^*} \hat{p}_{j,\ell}(\bn)e^{i\bn \cdot x} q_{\varepsilon,\ell }(x) \right]~ \cdot\chi\left(\frac{|\bxi|}{\varepsilon}<a \right)\\
&= \sum\limits_{\ell = 1,2}\sum\limits_{\bn \in \Lambda^*} \hat{p}_{j,\ell}(\bn) \mathcal{F} \left[  e^{i\bn \cdot x} q_{\varepsilon, \ell}(x) \right] ~ \cdot\chi\left(\frac{|\bxi|}{\varepsilon}<a \right)\\
&= \sum\limits_{\ell = 1,2}\sum\limits_{\bn \in \Lambda^*} \hat{p}_{j,\ell}(\bn) \hat{\alpha}_{\varepsilon, \ell} \left(\bxi - (\bK-\bk_{\rm D} + \bn)\right) ~ \cdot\chi\left(\frac{|\bxi|}{\varepsilon}<a \right) \numberthis \label{eq:gamma_sum} \\
\end{align*}
Note that by definition, $\hat{\alpha_{\varepsilon}}$ has compact support in the disc of radius $\varepsilon a $ around the origin. In the expansion above in \eqref{eq:gamma_sum}, consider first the case where $\bk_{\rm D}=\bK'$. Note that $\bK-\bk_{\rm D} = \bK-\bK' = 2\bK$ is not in the dual lattice $\Lambda ^*$. A term in the expression \eqref{eq:gamma_sum} is non-vanishing only if (i) $|\bxi-2\bK+\bn|<\varepsilon a$ (with $\bn\in\Lambda^*$) and 
 (ii) $|\bxi|<\varepsilon a$. For any $\bn\in\Lambda^*$ such that (i) and (ii) hold, we have  $|\bn-2\bK|-\varepsilon a \le |\bxi-2\bK+\bn|<\varepsilon a$. For $\varepsilon$ positive and sufficiently small, this implies that there are no such
   $\bn\in\Lambda^*$ because ${\rm dist}(2\bK,\Lambda^*)>0$. Hence, $\widehat{\gamma _{\varepsilon,\bK^\prime}}(\bxi)\equiv 0$ and therefore $\gamma _{\varepsilon,\bK^\prime}(\bx)\equiv 0$.

Next, consider the case $\bk_{\rm D}=\bK$. By similar reasoning, the only non-zero term in \eqref{eq:gamma_sum} arises from the lattice point $\bn = (0,0)$. Hence, 
\begin{align*}
\mathcal{F}[\gamma_{\varepsilon,\bK}]_j &= \sum\limits_{\ell = 1,2} \hat{p}_{j,\ell}({\bf 0}) \hat{\alpha}_{\varepsilon, \ell} \left(\bxi\right) ~ \cdot\chi\left(\frac{|\bxi|}{\varepsilon}<a \right) = \sum\limits_{\ell = 1,2} \int_\Omega \overline{\Phi_j(\by;\bK)}\Phi_l(\by;\bK) d\by\
 \hat{\alpha}_{\varepsilon, \ell}\left(\bxi\right)~ \cdot\chi\left(\frac{|\bxi|}{\varepsilon}<a \right)\\
 &= \hat{\alpha}_{\varepsilon, j}\left(\bxi\right)~ \cdot\chi\left(\frac{|\bxi|}{\varepsilon}<a \right)
 \end{align*}
 
 Summarizing, we have $\gamma _{\varepsilon,\bK}(\bx)\equiv \alpha_\varepsilon(\bx)$ 
  and $\gamma _{\varepsilon,\bK^\prime}(\bx)\equiv 0$. Substitution into \eqref{eq:Proj_BL} and recalling that 
   $u(\bx)= \Phi(\bx;\bK)^\top\alpha_\varepsilon(\bx)$ yields
   \begin{equation}\label{eq:Proj_BL1}
{\rm Proj}_{L^2(\mathbb{R}^2)} (|H^0-E_D|\leq 
\varepsilon)u = u(\bx) + \mathcal{O}(\varepsilon^3\|u\|_{L^2(\mathbb{R}^2)}) \, .
\end{equation}
This is equivalent to \eqref{wp_dk2}. The proof of Proposition \ref{prop:wp_dk} is now complete. }

\section{ Spectrum of the effective Dirac monodromy operator, ${\rm spec}(M_{\rm Dir})=S^1$;  Remark  \ref{rem:BL} }\label{sec:wkb}
Since the Dirac Hamiltonian $\slashed{D}(T)$, see \eqref{eq:diracA}, defines a flow which is unitary in $L^2(\R^2;\C^2)$, the Floquet multipliers
 (eigenvalues of the monodromy operator) must all lie on the unit circle, $S^1$. In this section we justify the discussion in Remark \ref{rem:BL} and prove that the spectrum of the monodromy operator 
associated with $\slashed{D}(T)$ is equal to $S^1$. To prove this, it suffices to show that the Floquet multipliers
 associated with $\widehat{\slashed{D}}(T;\bxi)$, see \eqref{eq:dirac_xi}, cover the unit circle as $|\bxi|$ varies outside of a sufficiently large circle in $\R^2_\xi$. 
 Specifically, we shall demonstrate this for $\bxi=(\xi,0)$ and $\xi\gg1$. As opposed to the analysis in Section \ref{sec:Dirac}, here we do not restrict $A(T)$ to the form \eqref{eq:A_circ}, but rather to any periodic and differentiable $A(T)$ with $\int_0^{\Tper} A_j (T) \, dT = 0$ for $j=1,2$.
Consider the one-parameter family of periodically forced systems;
\begin{equation}
i\partial_T \hat{\alpha}(T;\xi) = \widehat{\slashed{D}}(T;\bxi)\hat{\alpha}(T;\xi) = \left[\xi\sigma _1 + \usigma \cdot A(T)  \right]\hat{\alpha}(T;\xi) ,
\label{DFloq}\end{equation}
where $\usigma \cdot A = A_1\sigma_1 -A_2 \sigma_2$.   
In order to construct the monodromy matrix, we first construct, for $|\xi|\gg1$ a basis of solutions 
via a  WKB-type expansion. 
Set 
\begin{equation}\label{eq:wkbans}
\hat{\alpha}(T,\xi) = e^{-i\xi S_0(T)}\ B(T,\xi) \, ,
\end{equation} 
where $S_0(T)$ and $B(T,\xi)$ are to be determined.
Substitution into \eqref{DFloq} yields
\begin{equation}\label{eq:Dirac_SB}
\xi S_0^\prime(T)B(T,\xi) + i\partial_T B(T,\xi) = \left[ \xi\sigma_1 +\usigma \cdot A\right]B(T,\xi) \, .
\end{equation}
We expand $B(T,\xi)$ in powers of the small parameter $\xi^{-1}$:
\begin{align}
B(T,\xi) &= B_0(T) + \xi ^{-1} B_1 (T) + \xi ^{-2}B_2(T;\xi) \, , \label{eq:B_expansion} 
\end{align}
We next formally obtain equations for $B_0(T)$, $B_1(T)$ by equating terms of order $\xi$ and order $1$,
 and an equation for the corrector, $B_2(T,\xi)$ by balancing the remaining terms. This yields:
 
\begin{align}
\mathcal{O}(\xi^1):\quad \left(S_0^\prime(T) I\ -\ \sigma_1\right) B_0(T)  &= 0 \, , \label{eq:WKBo1}\\
\mathcal{O}(\xi^0):\quad \left(S_0^\prime(T) I\ -\ \sigma_1\right) B_1(T)  &= \usigma \cdot A(T) B_0(T) - i\partial_T B_0(T) \, , \label{eq:WKBo0}
\end{align}
and the equation for the corrector, $B_2(T,\xi)$:
\begin{equation}
i\partial_T B_2 + \xi \left(S_0^\prime(T)-\sigma_1\right) B_2 - \usigma \cdot A(T) B_2
= \xi\left(\ -i\partial_T + \usigma \cdot A(T) \right) B_1 \, . \label{B2}
\end{equation}
We next solve \eqref{eq:WKBo1}-\eqref{B2}. Equation  \eqref{eq:WKBo1} has a non-trivial solution $S_0^\prime(T)$ is an eigenvalue of~$\sigma_1$. 
Hence, we have the two linearly independent solutions
 \begin{align}
 S_{0+}^\prime(T)  &= +1,\qquad B_0(T)=c(T)\bw_+,\quad \bw_{+} = \frac{1}{\sqrt2}\begin{pmatrix}1\\ 1\end{pmatrix} \, ,
 \label{S0+}\\
  S_{0-}^\prime(T)  &= -1,\qquad B_0(T)=c(T)\bw_-,\quad \bw_{-} = \frac{1}{\sqrt2}\begin{pmatrix}1\\ -1\end{pmatrix} \, .
\label{S0-}    \end{align}
    We let $S_{0}^\prime(T)=\lambda$ denote either of these eigenvalues and $\bw_\lambda$ denote the corresponding normalized eigenvector.
  For this choice of eigenpair, \eqref{eq:WKBo0} becomes:
  \begin{equation}
\left(\lambda I\ -\ \sigma_1\right) B_1(T)  = \left(\ \usigma \cdot A(T) c(T) - i\partial_T c(T)\ \right)\bw_\lambda
 \label{B1}\end{equation}
Equation \eqref{B1} is solvable if and only if 
\[ \left\langle \bw_\lambda\ ,\ \left(\ \usigma \cdot A(T) c(T) - i\partial_T c(T)\ \right)\bw_\lambda\right\rangle=0\]
 or equivalently 
 \begin{equation} i\partial_T c(T) = \left\langle \bw_\lambda, \usigma \cdot A(T) \bw_\lambda\right\rangle c(T) .\label{c-eqn}\end{equation}
 This can be further simplified. We take $c(0)=1$ and assume $A_j(T)$ has zero mean on $[0,\Tper]$ for $j=1,2$. Note that $\usigma \cdot A \bw_\lambda = A_1\sigma_1\bw_\lambda - A_2\sigma_2\bw_\lambda$ and since $\sigma_2\bw_\lambda$ is orthogonal to $\bw_\lambda$, we obtain
 \begin{equation}
  i\partial_T c(T) = \left\langle \bw_\lambda,\sigma_1\bw_\lambda\right\rangle A_1(T) c(T) ,\quad 
  c(T) = \exp\left(-i\left\langle \bw_\lambda,\sigma_1\bw_\lambda\right\rangle \int_0^TA_1(s) ds\right)\ ,
  \label{cofT}
  \end{equation}
 where $\left\langle \bw_\lambda\ ,\sigma_1\bw_\lambda\right\rangle= \lambda $, where $\lambda=+1$ or $-1$.
 
 Next, using $B_0(T)=c(T)\bw_\lambda$ and \eqref{c-eqn},  \eqref{eq:o0} becomes
 \[ \left(\lambda I\ -\ \sigma_1\right) B_1(T)  = \left[ \usigma \cdot A(T) \bw_\lambda  - \left\langle \bw_\lambda, \usigma \cdot A(T) \bw_\lambda\right\rangle \bw_\lambda \right] c(T)  = \pi_\lambda^\perp\left[ \usigma \cdot A(T) \bw_\lambda \right] c(T) , 
 \]
 where $\pi_\lambda^\perp {\bf u}$ projects ${\bf u}\in \C^2$ onto the subspace orthogonal to ${\rm span}\{\bw_\lambda\}$, e.g., in the case of $\lambda = 1$, it projects onto the span of $\bw _{-1}$.
 Hence, we have:
 \begin{equation}
  B_1(T) = \left(\lambda I\ -\ \sigma_1\right)^{-1}\pi_\lambda^\perp\left[ \usigma \cdot A(T) \bw_\lambda \right] c(T),
  \label{B1a}
  \end{equation}
  up to an element in the kernel of $\lambda I\ -\ \sigma_1$ which we set equal to zero, and where $c(T)$ is given by~\eqref{cofT}.
  
  Summarizing, for each of the two eigenvalues $\lambda=S_0^\prime$ of $\sigma_1$ ($\lambda=\pm1$) 
   with corresponding normalized eigenvectors, $\bw_\lambda$, we have constructed first two terms of an  expansion
   of  $\widehat{\alpha}(T,\xi)=~e^{-i\lambda \xi T} B(T,\xi)$ by determining $B_0(T)$ (equations \eqref{S0+},\eqref{S0-},\eqref{c-eqn}) and $B_1(T)$ (equation \eqref{B1a}). Our expansion reads:
\begin{equation} \widehat{\alpha}_\lambda(T,\xi) =  e^{i\lambda\xi}\left[ c(T)\bw_\lambda + \frac{ c(T)}{\xi} 
   \left(\lambda I\ -\ \sigma_1\right)^{-1}\pi_\lambda^\perp\left[ \usigma \cdot A(T) \bw_\lambda \right] +\frac{1}{\xi^2} B_2(T,\xi) \right]\label{expand1}
   \end{equation}
 Next we bound $B_2(T,\xi)$ from the ODE \eqref{B2}. For any $T_0$ fixed  we obtain:
    \begin{equation}\label{eq:B2_ubd}
|B_2(T,\xi)|  \lesssim |\xi| T_0 \sup\limits_{0\le s\leq T_0}(|A(s)| + |A^\prime(s)| ) \, .
\end{equation}
By \eqref{expand1} and \eqref{eq:B2_ubd}
we have, corresponding to $\lambda=\pm1$, the pair of linearly independent solutions of \eqref{DFloq} given by:
\begin{equation} \widehat{\alpha}_\pm(T,\xi) = e^{\mp i\xi T}\left( c_\pm(T)\bw_+ + \mathcal{O}_\pm(|\xi|^{-1}) \right).\label{expand2}
   \end{equation}

Introduce the $2\times2$ matrix solution of \eqref{DFloq}, whose columns are $\widehat{\alpha}_+(T,\xi)$ and $\widehat{\alpha}_-(T,\xi)$:
\begin{equation}
V(T,\xi)=\Big[\widehat{\alpha}_+\ \Big| \ \widehat{\alpha}_-\Big](T,\xi) = 
\Big[ e^{ -i\xi T} c_+(T)\bw_+ + \mathcal{O}(|\xi|^{-1})\ \Big|\   e^{ i\xi T} c_-(T)\bw_-+ \mathcal{O}(|\xi|^{-1})\Big]\ .
\label{VTxi}\end{equation}
 Note $V(T+\Tper,\xi)= V(T,\xi)V^{-1}(0,\xi)V(\Tper,\xi)$. Hence the monodromy matrix for \eqref{DFloq} is given by:
 \[ M_{\rm Dir}(\xi) = V^{-1}(0,\xi)V(\Tper,\xi)  \, ,\]
with the compressed notation $M_{\rm Dir}((\xi,0))= M_{\rm Dir}(\xi)$. Since, by assumption, $c_\pm(0)=1$ and $A_1$ has zero mean, we have $c_\pm(\Tper)=1$.  Using \eqref{VTxi} we obtain: 
 $$M_D(\xi) = V_0^{-1}(0,\xi)V_0(\Tper, \xi) +\mathcal{O}_{2\times2}(\xi^{-1})  \, , \quad {\rm where}\quad  
 V_0(\Tper;\xi) = \Big[ e^{ -i\xi \Tper}\bw_+ \  \Big| \  e^{i\xi \Tper}\bw_- \Big]\ . $$
The $\mathcal{O}_{2\times2}(|\xi|^{-1})$ correction is with respect to any norm on complex $2\times2$ matrices. 
 Since $(e^{-i\xi \Tper}\bw_+ \big| e^{i\xi \Tper} \bw_- )=(\bw_+ \big| \bw_- )\ {\rm diag}( e^{-i\xi \Tper}, e^{i\xi \Tper})$  we have
\begin{align*}V_0^{-1}(0,\xi)V_0(\Tper, \xi) &= \left(\bw_+ \Big| \bw_- \right)^{-1}\left(e^{-i\xi \Tper}\bw_+ \Big| e^{i\xi \Tper} \bw_- \right)  = \left( \begin{array}{ll}
e^{-i\xi \Tper} & 0 \\ 0 & e^{i\xi \Tper}
\end{array} \right) \, .
\end{align*}
Therefore, 
$$M_D(\xi) = \left( \begin{array}{ll}
e^{-i\xi \Tper} & 0 \\ 0 & e^{i\xi \Tper}
\end{array} \right)  +\mathcal{O}_{2\times2}(|\xi|^{-1}) , $$
whose eigenvalues, $e^{\pm i\mu(\xi)}\in S^1$, can be computed and expressed as: $e^{\pm i\mu(\xi)}\equiv e^{\mp i\xi\Tper}+ \nu_\pm(\xi)$, where $|\nu_\pm(\xi)|\le C|\xi|^{-1}$, 
with a constant $C$ which depends on bounds on $A(T)$ and $A^\prime(T)$.\footnote{This is a particular case of the general statement in standard matrix perturbation theory, see e.g., \cite[Chapter IV, Theorem 1.1]{stewart1990matrix}).}  The mapping $\xi\mapsto e^{i\mu(\xi)}$
  covers  $S^1$.  Indeed, for $\eta_0>0$ there exists $N_0$ large such that the image  of  the closed interval $[-\eta_0, 2\pi+\eta_0] + 2\pi N_0/\Tper$ under $e^{i\mu(\xi)}$ is  dense in $S^1$. Furthermore, by continuity, the image of this interval is closed
  and hence equal to $S^1$.
It follows that
  $$S^1=\cup_{|\xi|\ge\xi_0}{\rm spec}(M_{\rm Dir}(\xi)) \subseteq {\rm spec}(M_{\rm Dir})\subseteq S^1 \, .$$  Hence ${\rm spec}(M_{\rm Dir})=S^1$.
   This completes the justification of 
  \eqref{eq:Diracspec_infty} in Remark \ref{rem:BL}.

\appendix

\section{The Spectral Theorem for Unitary Operators}\label{ap:spectral}
We review here in short the basic elements and formulations of the spectral theory of unitary operators on Hilbert spaces, see details in e.g., \cite{hall2013quantum, taylor2013partial}. 

\begin{definition}[projection-valued measure]
Let $\mathcal{H}$ be a Hilbert space, let $X$ be a set, and $\Sigma$ a $\sigma$-algebra in $X$. A map $\Pi:\Sigma \to B(\mathcal{H})$, the Banach space of bounded linear operators on $\mathcal{H}$, is called a {\em projection-valued measure} if the following properties hold:
\begin{enumerate}
\item $\Pi(I)$ is an orthogonal projection for every $I\in \Sigma$.
\item $\Pi(\emptyset)=0$ and $\Pi(X) = {\rm Id}$.
\item If $\{I_j\}_{j\geq 1} \subset \Sigma$ are disjoint then $$\Pi \left( \bigcup\limits_{j\geq 1} I_j\right) v = \sum\limits_{j\geq 1} \Pi(I_j) v \, , \qquad v\in \mathcal{H} \, .$$
\item $\Pi (I_1 \cap I_2) = \Pi (I_1)\Pi(I_2)$ for all $I_1, I_2 \in \Sigma$.
\end{enumerate}
\end{definition}

 \begin{theorem}
Let $U$ be a unitary operator on $\mathcal{H}$. There exists a unique projection-valued measure $\Pi = \Pi(\cdot; U)$ on the Borel $\sigma$-algebra of $S^1$ such that 
$$\int\limits_{S^1} z \, d\Pi(z) = U \, .$$

\end{theorem}

\bibliographystyle{siam}
\bibliography{floquetBib}


\newpage

\end{document}